\documentclass[12pt]{amsart}
\usepackage{amsfonts,amssymb,amsbsy,amsmath,amsthm}
\newtheorem{theorem}{Theorem}[section]
\newtheorem{lemma}{Lemma}[section]
\newtheorem{follow}{Corollary}[section]
\newtheorem{pr}{Proposition}[section]
\theoremstyle{definition}

\newcommand{\bel}{\begin{equation} \label}
\newcommand{\ee}{\end{equation}}

\newcommand{\xp}{X_\perp}
\newcommand{\ep}{\varkappa}

\newcommand{\rt}{{\mathbb R}^{3}}
\newcommand{\rd}{{\mathbb R}^{2}}
\newcommand{\re}{{\mathbb R}}

\newcommand{\gaden}{\tilde{I}_{\perp} \otimes H_{0, \parallel}}
\newcommand{\ph}{\chi(\hm)}
\newcommand{\pho}{\chi(\hmo)}
\newcommand{\hopa}{H_{0, \parallel}}
\newcommand{\hpa}{H_{\parallel}}
\newcommand{\ppar}{p_{\parallel}}
\newcommand{\hope}{H_{0, \perp}}
\newcommand{\hompe}{H_{0, \perp}^{(m)}}
\newcommand{\hmo}{H_{0}^{(m)}}
\newcommand{\hm}{H^{(m)}}
\newcommand{\mpe}{m_{\perp}}
\newcommand{\lre}{L^2_{\rm Re}(\re_+; \varrho d\varrho)}

\newcommand{\Z}{{\mathbb Z}}

\newcommand{\C}{{\mathbb C}}

\newcommand{\Tr}{\operatorname{Tr}}

\newcommand{\tihm}{\tilde{H}^{(m)}}

{

\title[Resonances and SSF singularities for a magnetic Schr\"odinger operator]
{Dynamical resonances and SSF singularities for a magnetic
Schr\"odinger operator}
\author[M. A. Astaburuaga]
{Mar\'ia Ang\'elica Astaburuaga}
\address{Departamento de Matem\'aticas, Facultad de Matem\'aticas, Pontificia Universidad
Cat\'olica de Chile, Vicu\~na Mackenna 4860, Santiago de Chile}
\email{angelica@mat.puc.cl}
\author[Ph. Briet]{Philippe Briet }
\address{Centre de Physique Th\'eorique,
CNRS-Luminy, Case 907, 13288 Marseille,
France}\email{briet@cpt.univ-mrs.fr}
\author[V. Bruneau]{Vincent Bruneau }
\address {Universit\'e Bordeaux I, Institut de Math\'ematiques de Bordeaux,  UMR CNRS 5251,  351, Cours de la Lib\'eration, 33405
Talence, France}\email{vbruneau@math.u-bordeaux1.fr}
\author[C. Fern\'andez]
{Claudio Fern\'andez}
\address{Departamento de Matem\'aticas, Facultad de Matem\'aticas, Pontificia Universidad
Cat\'olica de Chile, Vicu\~na Mackenna 4860, Santiago de Chile}
\email{cfernand@mat.puc.cl}
\author[G. Raikov]
{Georgi Raikov}
\address{Departamento de Matem\'aticas, Facultad de Matem\'aticas, Pontificia Universidad
Cat\'olica de Chile, Vicu\~na Mackenna 4860, Santiago de Chile}
\email{graikov@mat.puc.cl}

\begin{document}

\maketitle
\begin{center}
{\em \large Dedicated to Vesselin Petkov on the occasion of his
65th birthday}\\
\end{center}

\begin{abstract}
We consider the Hamiltonian $H$ of a 3D spinless non-relativistic
quantum particle subject to parallel constant magnetic and
non-constant electric field. The operator $H$ has infinitely many
eigenvalues of infinite multiplicity embedded in its continuous
spectrum. We perturb $H$ by appropriate scalar potentials $V$ and
investigate the transformation of these embedded eigenvalues into
resonances. First, we assume that the electric potentials are
dilation-analytic with respect to the variable along the magnetic
field, and obtain an asymptotic expansion of the resonances as the
coupling constant $\varkappa$ of the perturbation tends to zero.
Further, under the assumption that the Fermi Golden Rule holds
true, we deduce estimates for the time evolution of the resonance
states with and without analyticity assumptions; in the second
case we obtain these results as a corollary of suitable Mourre
estimates and a recent article of Cattaneo, Graf and Hunziker
\cite{cgh}. Next, we describe sets of perturbations $V$ for which
the Fermi Golden Rule is valid at each embedded eigenvalue of $H$;
these sets turn out to be dense in various suitable topologies.
Finally, we assume that $V$ decays fast enough at infinity and is
of definite sign, introduce the Krein spectral shift function for
the operator pair $(H+V, H)$, and study its singularities at the
energies which coincide with eigenvalues of infinite multiplicity
of the unperturbed operator $H$.
\end{abstract}

\bigskip

{\bf  2000 AMS Mathematics Subject Classification}:  35P25, 35J10,
47F05, 81Q10
\bigskip

{\bf Keywords}: magnetic Schr\"odinger operators, resonances,
Mourre estimates, spectral shift function
\bigskip


\vspace{1cm}

\section{Introduction}
\setcounter{equation}{0} In the present article we consider a
magnetic Schr\"odinger operator $H$ which, from a physics point of
view, is the quantum Hamiltonian of a 3D non-relativistic spinless
quantum particle subject to an electromagnetic field $({\bf E},
{\bf B})$ with electric component ${\bf E} = -(0,0,v_0')$ where
$v_0$ is a scalar potential depending only on the variable $x_3$,
and magnetic component ${\bf B} = (0,0,b)$ where $b$ is a positive
constant. From a mathematical  point of view this operator is
remarkable because of the generic presence of infinitely many
eigenvalues of infinite multiplicity, embedded in the continuous
spectrum of $H$. These eigenvalues have the form $2bq + \lambda$,
$q \in {\mathbb Z}_+ : = \{0,1,2,\ldots\}$, where $2bq$, $q \in
\Z_+$, are the Landau levels, i.e. the infinite-multiplicity
eigenvalues of the (shifted) Landau Hamiltonian, and $\lambda$ is
a simple eigenvalue of the 1D operator $-\frac{d^2}{dx^2} +
v_0(x)$.  We introduce the perturbed operator $H + \varkappa V$
where $V$ is a $H$-compact multiplier by a real function, and
$\varkappa \in \re$ is a coupling constant, and study the
transition of the eigenvalues $2bq + \lambda$, $q \in {\mathbb
Z}_+$, into a ``cloud" of resonances which converge to $2bq +
\lambda$ as
$\varkappa \to 0$.\\
 In order to perform this analysis, we assume
that $V$ is axisymmetric so that the operator $H + \varkappa V$
commutes with the $x_3$-component of the angular-momentum
operator $L$. In this case $H + \varkappa V$ is unitarily
equivalent to the orthogonal sum $\oplus_{m \in \Z} (\hm +
\varkappa V)$ where $\hm$ is unitarily equivalent to the
restriction of $H$ onto ${\rm Ker}\,(L-m)$, $m \in \Z$. This
allows us to reduce the analysis to a perturbation of a simple
eigenvalue $2bq + \lambda$ of the operator $\hm$ with fixed
magnetic quantum number $m$.\\
We apply two different approaches to the definition of resonances.
First, we suppose that $v_0$ and $V$ are analytic in $x_3$, and
following the classical approach of Aguilar and Combes
\cite{AgCo},  define the resonances as the eigenvalues of the
dilated non-self-adjoint operator $H(\theta) + \varkappa
V_{\theta}$. We obtain an asymptotic expansion as $\varkappa \to
0$ of each of these resonances in the spirit of the Fermi Golden
Rule (see e.g. \cite[Section XII.6]{RSIV}), and estimate the time
decay of the resonance states. A similar relation between the
small-coupling-constant asymptotics of the resonance, and the
exponential time decay of the resonance state has been established
by Herbst \cite{h} in the case of the Stark Hamiltonian, and later
by other authors in the case of various quantum Hamiltonians (see
e.g.
\cite{sk}, \cite{fk}, \cite{aabcdf}). \\
Our other approach is close to the time dependent methods
developed in \cite{sw} and \cite{cs}, and, above all, to the
recent article by Cattaneo, Graf and Hunziker \cite{cgh}, where
the dynamic estimates of the resonance states are based on
appropriate Mourre estimates \cite{Mo}. We prove Mourre type
estimates for the operators $\hm$, which might be of independent
interest, and formulate a theorem on the dynamics of the resonance
states which can be regarded as an application of the general
abstract result of \cite{cgh}.\\
Both our approaches are united by the requirement that the
perturbation $V$ satisfies the Fermi Golden Rule for all embedded
eigenvalues for the operators $\hm$, $m \in \Z$. We establish
several results which show that the set of such perturbations is
dense in various topologies compatible with the assumptions of our
theorems on the resonances of $H + \varkappa V$.\\
Further, we cancel the restriction that $V$ is axisymmetric but
suppose that it decays fast enough at infinity, and has a definite
sign, introduce the Krein spectral shift function (SSF) for the
operator pair $(H+V,H)$, and study its singularity at each energy
$2bq + \lambda$, $q \in \Z_+$, which, as before, is an eigenvalue
of infinite multiplicity of the unperturbed operator $H$. We show
that the leading term of this singularity can be expressed via the
eigenvalue counting function for compact Berezin-Toeplitz
operators. Using the well-known results on the spectral
asymptotics for such operators (see \cite{r0}, \cite{rw}), we
obtain explicitly the main asymptotic term of the SSF as the
energy approaches the fixed point $2bq + \lambda$ for several
classes of perturbations with prescribed decay rate with respect
to the variables on the plane perpendicular to the
magnetic field. \\
It is natural to associate these singularities of the SSF to the
accumulation of resonances to these points because it  is
conjectured that the resonances are the poles of the SSF. This
conjecture is justified by the Breit-Wigner approximation which is
mathematically proved in other cases (see
for instance \cite{pz1}, \cite{pz2}, \cite{bp}, \cite{bbr}).\\
The article is organized as follows. In Section 2 we summarize
some well-known spectral properties of the operators $H$ and $\hm$
and their perturbations, which are systematically exploited in the
sequel. Section 3 is devoted to our approach based on the dilation
analyticity, while Section 4 contains our results obtained as
corollaries of appropriate Mourre estimates. In Section 5 we
discuss the density in suitable topologies of the sets of
perturbations $V$ for which the Fermi Golden Rule holds true for
every embedded eigenvalue $2bq + \lambda$ of the operator $\hm$,
$m \in \Z$. Finally, the asymptotic analysis of the SSF
near the points $2bq + \lambda$ can be found in Section 6. \\
We dedicate the article to Vesselin Petkov with genuine admiration
for his most significant contributions to the spectral and
scattering theory for partial differential operators. In
particular, we would like to mention his keystone results on the
distribution of resonances, and the Breit-Wigner approximation of
the spectral shift function for various quantum Hamiltonians (see
\cite{pz1}, \cite{pz2}, \cite{bp}), and, especially, his recent works on
magnetic Stark operators (see \cite{dp1}, \cite{dp2}). These
articles as well as many other works of Vesselin have
strongly influenced and stimulated our own research.\\

\section{Preliminaries}
\setcounter{equation}{0}

{\bf 2.1.} In this subsection we summarize some well-known facts
on the spectral properties of the 3D Schr\"odinger operator with
constant magnetic field ${\bf B} = (0,0,b)$, $b = {\rm const.} >
0$. More details could be found, for example, in \cite{ahs} or
\cite[Section 9]{BPR}.\\
Let
$$
H_0 : = \hope \otimes I_{\parallel} + I_{\perp} \otimes \hopa
$$
where  $I_{\parallel}$ and $I_{\perp}$ are the identity operators
in $L^2(\re_{x_3})$ and $L^2(\rd_{x_1, x_2})$ respectively,
$$
\hope : = \left(i\frac{\partial}{\partial x_1} -
\frac{bx_2}{2}\right)^2 + \left(i\frac{\partial}{\partial x_2} +
\frac{bx_1}{2}\right)^2 - b, \quad (x_1,x_2) \in \rd,
$$
is the Landau Hamiltonian shifted by the constant $b$, and
$$
\hopa : = -\frac{d^2}{d x_3^2}, \quad x_3 \in \re.
$$
The operator $\hope$ is self-adjoint in $L^2(\rd)$, the operator
$\hopa$ is self-adjoint in $L^2(\re)$, and hence the operator
$H_0$ is self-adjoint in $L^2(\re^3)$. Moreover, we have
$\sigma(\hope) = \cup_{q=0}^{\infty} \{2bq\}$, and every
eigenvalue $2bq$ of $\hope$ has infinite multiplicity (see e.g.
\cite{ahs}).
 Denote by $p_q$ the orthogonal
projection onto ${\rm Ker}\,(\hope -2bq)$, $q \in {\mathbb Z}_+$.
Since $\sigma(\hopa) = [0,\infty)$, we have $\sigma(H_0) =
\cup_{q=0}^{\infty} [2bq, \infty) =
[0,\infty)$.\\
Let now $m \in {\mathbb Z}$, $\varrho = (x_1^2 + x_2^2)^{1/2}$. Put
$$
\hompe : = - \frac{1}{\varrho} \frac{\partial}{\partial \varrho} \varrho
\frac{\partial}{\partial \varrho} + \left(\frac{m}{\varrho} -
\frac{b\varrho}{2}\right)^2 - b.
$$
The operator $\hompe$ is self-adjoint in $L^2(\re_+; \varrho d\varrho)$,
and we have $\sigma(\hompe) = \cup_{q =m_-}^{\infty} \{2bq\}$
where $m_- = \max\{0, -m\}$ (see e.g. \cite{ahs}). In contrast to
the operator $\hope$, every eigenvalue $2bq$ of $\hompe$ is
simple. Denote by $\tilde{p}_{q,m}$ the orthogonal projection onto
${\rm Ker}\,(\hompe - 2bq)$, $q \in {\mathbb Z}_+$, $q \geq m_-$.
Put
    \bel{ser1}
\varphi_{q,m}(\varrho) : = \sqrt{\frac{q!}{\pi
(q+m)!}\left(\frac{b}{2}\right)^{m+1}} \varrho^m
L_q^{(m)}\left(b\varrho^2/2\right) e^{-b \varrho^2/4}, \quad
\varrho \in \re_+, \quad q \in {\mathbb Z}, \quad q \geq m_-,
    \ee
where
    \bel{x3}
L_q^{(m)}(s) : = \sum_{l=m_-}^{q} \frac{(q+m)!}{(m+l)! (q-l)!}
\frac{(-s)^l}{l!}, \quad s \in \re,
    \ee
are the generalized Laguerre polynomials. Then we have
$$
\hompe \varphi_{q,m} = 2bq \varphi_{q,m},
$$
$\|\varphi_{q,m}\|_{L^2(\re_+; \varrho d\varrho)} = 1$, and
$\varphi_{q,m} = \overline{\varphi_{q,m}}$ (see e.g. \cite[Section
9]{BPR}). Moreover, $\tilde{p}_{q,m} = |\varphi_{q,m}\rangle
\langle\varphi_{q,m}|$.
\\
Set
$$
\hmo : = \hompe \otimes I_{\parallel} + \tilde{I}_{\perp} \otimes
\hopa
$$
where $\tilde{I}_{\perp}$ is the identity operator in $L^2(\re_+;
\varrho d\varrho)$. Evidently, $\sigma(\hmo) = [2m_-b,
\infty)$.\\
Let $(\varrho, \phi, x_3)$ be the cylindrical coordinates  in
$\re^3$. The operator $\hmo$, $m \in {\mathbb Z}$, is unitarily
equivalent to the restriction of $H_0$ onto ${\rm Ker}\,(L-m)$
where
$$
L : = -i\left(x_1 \frac{\partial}{\partial x_2} - x_2
\frac{\partial}{\partial x_1}\right) = -i \frac{\partial}{\partial
\phi}
$$
is the $x_3$-component of the angular-momentum operator, which
commutes with $H_0$. \\
 Moreover, the operator $H_0$ is unitarily equivalent to the
orthogonal sum $\oplus_{m \in {\mathbb Z}} \hmo$. More precisely,
if we pass to cylindrical coordinates, and decompose $u \in {\rm
Dom}(H_0)$ into a Fourier series with respect to $\phi$, i.e. if
we write
$$
u(\varrho \cos{\phi}, \varrho \sin{\phi}, x_3) = \sum_{m \in
{\mathbb Z}} e^{im\phi} u_m(\varrho, x_3),
$$
we have
$$
(H_0 u)(\varrho \cos{\phi}, \varrho \sin{\phi}, x_3) = \sum_{m \in
{\mathbb Z}} e^{im\phi} (\hmo u_m)(\varrho, x_3).
$$

{\bf 2.2.} In this subsection we perturb the operators $\hmo$ and
$H_0$ by a scalar potential $v_0$ which depends only on the
variable $x_3$.\\  Let $v_0 : \re \to \re$ be a measurable
function. Throughout the paper we assume that the multiplier by
$v_0$ is $\hopa$-compact, which is ensured, for instance, by $v_0
\in L^2(\re) + L^{\infty}_{\varepsilon}(\re)$. Set
$$
\hpa : = \hopa + v_0.
$$
Then we have
$$
\sigma_{\rm ess}(\hpa) = \sigma_{\rm ess}(\hopa) = [0, \infty).
$$
For simplicity, throughout the article we suppose also that
    \bel{ser7}
    \inf{\sigma(\hpa)} > - 2b.
    \ee
Evidently, \eqref{ser7} holds true if the negative part $v_{0,-}$
of the function $v_0$ is bounded, and we have
$\|v_{0,-}\|_{L^{\infty}(\re)} < 2b$. \\
Assume now that the discrete spectrum of $\hpa$ is not empty; this
would follow, for example, from the additional conditions $v_0 \in
L^1(\re)$ and $\int_{\re} v_0(x) dx < 0$ (see e.g. \cite[Theorem
XIII.110]{RSIV}). Occasionally, we will impose also the assumption
that the discrete spectrum of $\hpa$ consists of a unique
eigenvalue; this would be implied, for instance, by the inequality
$\int_{\re}|x|v_{0,-}(x) dx < 1$ (see
e.g. \cite[Chaper II, Theorem 5.1]{bs}). \\
 Let $\lambda$ be a discrete eigenvalue of the operator $\hpa$
 which necessarily is simple.  Then  $\lambda \in (-2b,0)$ by
\eqref{ser7}. Let $\psi$ be an eigenfunction satisfying
    \bel{ser2}
    \hpa \psi = \lambda \psi, \quad \|\psi\|_{L^2(\re)} = 1,
    \quad \psi = \overline{\psi}.
    \ee
    Denote by $\ppar = \ppar(\lambda)$ the spectral projection onto
    ${\rm Ker}(\hpa - \lambda)$. We have $\ppar = |\psi\rangle
    \langle \psi|$.\\
    Suppose now that $v_0$ satisfies
    \bel{31}
    |v_0(x)| = O \left(\langle x \rangle^{-m_0}\right), \quad x \in \re,
    \quad m_0 > 1,
    \ee
    where $\langle x \rangle:= (1 + | x |^2)^{\frac12}$.
Then the multiplier by $v_0$ is a relatively trace-class
perturbation of $\hopa$, and by the Birman-Kuroda theorem (see
e.g. \cite[Theorem XI.9]{RSIII}) we have
$$
\sigma_{\rm ac}(\hpa) = \sigma_{\rm ac}(\hopa) = [0,\infty).
$$
Moreover, by the Kato theorem (see e.g. \cite[Theorem
XIII.58]{RSIV}) the operator $\hpa$ has no strictly positive
eigenvalues.   In fact, for all $E > 0$ and $s>1/2$ the
operator-norm limit
    \bel{35}
    \langle x \rangle^{-s} (\hpa - E)^{-1} \langle x \rangle^{-s} : =
    \lim_{\delta \downarrow 0} \langle x\rangle^{-s}  (\hpa - E - i\delta)^{-1} \langle x \rangle^{-s},
    \ee
    exists in ${\mathcal L}(L^2(\re))$, and for each compact subset
    $J$ of $\re_+ = (0,\infty)$ and each $s > 1/2$
    there exists a constant $C_{J,s}$ such that for each $E \in J$
    we have
    \bel{36}
\|\langle x \rangle^{-s} (\hpa  - E)^{-1} \langle x \rangle^{-s}\|
\leq C_{J,s}
    \ee
(see \cite{A}).\\
Suppose again that \eqref{31} holds true, and let us consider the
differential equation
    \bel{fgr1}
    - y'' + v_0 y = k^2 y, \quad k \in \re.
    \ee
    It is well-known that \eqref{fgr1} admits the so-called
    {\em Jost solutions} $y_1(x;k)$ and $y_2(x;k)$ which obey
    $$
    y_1(x;k) = e^{ikx} (1 + o(1)),
    \quad x \to \infty,
    $$
    $$
    y_2(x;k) = e^{-ikx}(1 + o(1)),
    \quad x \to -\infty,
    $$
    uniformly with respect to $k \in \re$ (see e.g. \cite[Chapter II, Section 6]{bs}
    or \cite{tz}). The
    pairs $\{y_l(\cdot;k), y_l(\cdot; -k)\}$, $k \in \re$, $l=1,2$, form
    fundamental sets of solutions of \eqref{fgr1}. Define the {\em
    transition coefficient} ${\mathcal T}(k)$ and {\em the
    reflection coefficient} ${\mathcal R}(k)$, $k \in \re$, $k
    \neq 0$, by
    $$
    y_2(x;k) = {\mathcal T}(k) y_1(x;-k) + {\mathcal R}(k)
    y_1(x;k), \quad x \in \re.
    $$
    It is well known that ${\mathcal T}(k) \neq 0$, $k \in
    \re\setminus\{0\}$. On the other hand, the Wronskian of the
    solutions $y_1(\cdot; k)$ and $y_2(\cdot; k)$ is equal to
    $-2ik{\mathcal T}(k)$, and hence these solutions are linearly
    independent for $k \in {\re}\setminus\{0\}$.
For $E>0$ set
$$
\Psi_l(x; E) : = \frac{1}{\sqrt{4\pi \sqrt{E}} {\mathcal
T}(\sqrt{E})} y_l(x;\sqrt{E}), \quad l=1,2.
$$
Evidently, $\Psi_l(\cdot;E) \in C^{1}(\re) \cap L^{\infty}(\re)$,
$E>0$, $l=1,2$. Moreover,  the real and the imaginary part of both
functions $\Psi_l(\cdot; E)$,
    $l= 1,2$ with $E > 0$ do not vanish identically.  Further,
     ${\rm Im}\,\langle x \rangle^{-s} (\hpa - E)^{-1}\langle
x \rangle^{-s}$ with $E > 0$ and $s>1/2$ is a rank-two operator
with an integral kernel
$$
K(x,x') = \pi \sum_{l=1,2}\langle x \rangle^{-s} \Psi_l(x; E)
\overline{\Psi}_l(x'; E)\langle x' \rangle^{-s}, \quad x,x' \in
\re.
$$
{\bf 2.3.} Fix now $m \in {\mathbb Z}$, and set
$$
\hm : = \hmo + \tilde{I}_{\perp} \otimes v_0.
$$
Since $v_0$ is $\hopa$-compact and the spectrum of $\hompe$ is
discrete, the operator $\tilde{I}_{\perp} \otimes v_0$ is
$\hmo$-compact. Therefore, the operator $\hm$ is well-defined on
${\rm Dom}(\hmo)$, and we have
$$
\sigma_{\rm ess}(\hm) = \sigma_{\rm ess}(\hmo) = [2bm_-, \infty).
$$
Further, if $\lambda$ is a discrete eigenvalue of $\hpa$, then
$2bq + \lambda$ is a simple eigenvalue of $\hm$ for each integer
$q \geq m_-$. If $q = m_-$, then this eigenvalue is isolated, but
 if $q > m_-$, then due to \eqref{ser7}, it is embedded in the
essential spectrum of $\hm$. Moreover,
    \bel{vin55}
\hm \Phi_{q,m} = (2bq + \lambda) \Phi_{q,m}, \quad q \geq m_-,
    \ee
$\Phi_{q,m} = \varphi_{q,m} \otimes \psi$, $\varphi_{q,m}$ being
defined in \eqref{ser1}, and $\psi$ in \eqref{ser2}. Set
    \bel{vin50}
    {\mathcal P}_{q,m} : = \tilde{p}_{q,m} \otimes \ppar.
    \ee
    Then we have ${\mathcal P}_{q,m} = |\Phi_{q,m}\rangle
    \langle\Phi_{q,m}|$. \\
    Finally, introduce the operator
    $$
    H : = H_0 + I_{\perp} \otimes v_0.
    $$
Even though the operator $I_{\perp} \otimes v_0$ is not
$H_0$-compact (unless $v_0 = 0$), it is $H_0$-bounded with zero
relative bound so that the operator $H$ is well-defined on ${\rm
Dom}(H_0)$. Evidently, the operator $H$ is unitarily equivalent to
the orthogonal sum $\oplus_{m \in {\mathbb Z}} \hm$. Up to the
additive constant $b$, the operator $H$ is the Hamiltonian of a
quantum non-relativistic spinless particle in an electromagnetic
field $({\bf E}, {\bf B})$ with parallel electric component ${\bf
E} = - (0,0,v_0'(x_3))$, and magnetic component ${\bf B}
= (0,0,b)$. \\
Note that if $\lambda$ is a discrete eigenvalue of $\hpa$, then
$2bq + \lambda$, $q \in {\mathbb Z}_+$ is an eigenvalue of
infinite multiplicity of $H$. If $q = 0$, this eigenvalue is
isolated, and if $q \geq 1$, it  lies on the interval $[0,
\infty)$ which constitutes a part of the essential spectrum of
$H$. Moreover, if \eqref{31} holds, then $\sigma_{\rm ac}(H) =
[0,\infty)$, so that in this case
$2bq + \lambda$, $q \in {\mathbb Z}_+$, is embedded in the absolutely continuous spectrum of $H$. \\

{\bf 2.4.} In this subsection we introduce appropriate
perturbations of the operators $H$ and $\hm$, $m \in {\mathbb
Z}_+$. \\
Let $V: \re^3 \to \re$ be a measurable function. Assume that $V$
is $H_0$-bounded with zero relative bound. By the diamagnetic
inequality (see e.g. \cite{ahs}) this would follow, for instance,
from $V \in L^2(\re^3) + L^{\infty}(\re^3)$. On ${\rm Dom}(H) =
{\rm Dom}(H_0)$ define the operator $H + \varkappa V$, $\varkappa
\in \re$.\\
{\em Remark}: We impose the condition that the relative
$H_0$-bound is zero just for the sake of simplicity. If $V$ is
$H_0$-bounded with arbitrary finite relative bound, then we could
again define
$H + \varkappa V$ but only for sufficiently small $|\varkappa|$.\\
Occasionally, we will impose the more restrictive assumption that
$V$ is $H_0$-compact; this would follow from $V \in L^2(\re^3) +
L^{\infty}_{\varepsilon}(\re^3)$. In particular, $V$ is
$H_0$-compact if it satisfies the estimate
    \bel{33}
    |V({\bf x})| = O \left(\langle \xp \rangle^{-\mpe} \langle x_3
    \rangle^{-m_3}\right), \; {\bf x} = (\xp, x_3),  \; \mpe > 0, \; m_3 > 0.
    \ee
 \\
 Further, assume that $V$ is axisymmetric, i.e. $V$ depends only on
 the variables $(\varrho, x_3)$. Fix $m \in \Z$ and assume that the
multiplier by $V$ is $\hmo$-bounded with zero relative bound. Then
the operator $\hm + V$ is well defined on ${\rm Dom}(\hm) = {\rm
Dom}(\hmo)$.
Define the operator $\hm + \varkappa V$, $\varkappa \in \re$.\\
For $z \in {\mathbb C}_+ : = \{\zeta \in {\mathbb C}\,|\,{\rm
Im}\,\zeta
> 0\}$, $m \in \Z$, $q \geq m_-$,  introduce the quantity
$$
F_{q,m}(z) : = \langle   (\hm - z)^{-1} (I-{\mathcal P}_{q,m}) V
\Phi_{q,m}, V\Phi_{q,m} \rangle
$$
where $\langle \cdot, \cdot \rangle$ denotes the scalar product in
$L^2(\re_+ \times \re; \varrho d\varrho dx_3)$, which we define to
be linear with respect to the first factor. If $\lambda$ is a
discrete eigenvalue of $\hpa$  we will say that {\em the Fermi
Golden Rule} ${\mathcal F}_{q,m,\lambda}$ is valid if the limit
    \bel{ser4}
F_{q,m}(2bq+\lambda) = \lim_{\delta \downarrow 0}
F_{q,m}(2bq+\lambda + i\delta),
    \ee
exists and is finite, and
    \bel{ser4a}
{\rm Im}\,F_{q,m}(2bq+\lambda) > 0.
    \ee

\section{Resonances via dilation analyticity}
\setcounter{equation}{0} {\bf 3.1.} In this subsection we will
perturb $\hmo$ by an axisymmetric potential $V(\varrho,x_3)$ so
that the simple eigenvalue $2bq + \lambda$ of $\hmo$  becomes a
resonance of the perturbed operator. In order to use complex
scaling, we impose an analyticity assumption. We assume that the
potential $v_0$  extends to an analytic function in the sector
$$S_{\theta_0}= \{z \in \C \; | \; \vert \hbox{Arg} z \vert  \leq \theta_0, \;
\hbox{or} \; \vert z \vert \leq r_0\}$$ with $\theta_0 \in
(0,\pi/2)$, which satisfies \eqref{31}. As already used in similar
situations (see e.g. \cite{ahs}, \cite{wang}) we introduce complex
deformation in the longitudinal variable, $({\mathcal U}(\theta)
f)(\varrho,x_3)= e^{\theta/2} f(\varrho, e^{\theta }x_3)$, $f \in
L^2(\re_+ \times \re; \varrho d\varrho dx_3)$, $\theta \in \re$.
For $\theta \in \re$ we have
$$\hm(\theta)=  {\mathcal U}(\theta)\hm {\mathcal U}^{-1}(\theta)= \hmo \otimes I + I \otimes H_{\parallel}(\theta),$$
with $H_{\parallel}(\theta)= -e^{-2 \theta} \frac{d^2}{dx_3^2} +
v_{0,\theta}(x_3)$, and $v_{0,\theta}(x_3) : = v_0(e^\theta
\,x_3)$.
By assumption, the  family of operators $ \{\hpa(\theta), \quad
\vert \hbox{Im } \theta \vert < \theta_0 \}$, form a type (A)
analytic family of $m$-sectorial operators in the sense of Kato
(see for instance \cite[Section 15.4]{HiSi}, \cite{AgCo}). Then
the discrete spectrum of $H_{0,\parallel}(\theta)$ is independent
of $\theta$ and we have
$$
\sigma(\hm(\theta)) = \bigcup_{q\geq m_-}\{
2bq+\sigma(\hpa(\theta))\},
$$
$$
\sigma(\hpa(\theta)) =
e^{-2 \theta} \re_+ \cup\sigma_{\rm disc}(\hpa)  \cup \{z_1, z_2,
....\}
$$
where  $\sigma_{\rm disc}(\hpa)$
 denotes the discrete spectrum of $\hpa$, and  $z_1, z_2, .... $ are (complex) eigenvalues
 of $H_{\parallel}(\theta)$ in
 $  \{0   > \hbox{Arg}\,z > -2 \hbox{Im }\theta\}$, $\hbox{Im } \theta >0$.
 In the sequel, we assume that
$\sigma_{\rm disc}(\hpa)=\{\lambda\}$.\\
Further, we assume that $V$ is axisymmetric, and admits an
analytic extension with respect to $x_3 \in S_{\theta_0}$, which
is $\hmo$-compact (see e.g. \cite[Chapter XII]{RSIV}). Let
$$ V_{\theta}(\varrho,x_3) : = V(\varrho,e^\theta x_3).$$
  Then the family of operators $ \{\hm(\theta) + \ep V_{\theta}, \quad
\vert \hbox{Im } \theta \vert < \theta_0, \; |\ep| \leq 1 \}$,
form also  an analytic family of type (A)  for sufficiently small
$\ep$.\\
By definition, the resonances of $\hm + \ep V $ in
$${\mathcal S}_{m_-}(\theta):= \bigcup_{q \geq m_-} \{ z \in \C; \;2bq< \hbox{Re} z < 2b(q+1),
\; -2 \hbox{Im} \theta < \hbox{Arg} (z-2bq)   \leq 0 \}$$
 are the eigenvalues of $\hm(\theta) + \ep V_{\theta}$, $ \hbox{Im} \theta
 >0$.\\
For $V$ axisymmetric and $H_0$-compact, we define the set of the
resonances $\hbox{Res}(H + \ep V, {\mathcal S}_{0}(\theta))$ of
the operator $H + \ep V$ in ${\mathcal S}_{0}(\theta)$ by
$$
\hbox{Res}(H + \ep V, {\mathcal S}_{0}(\theta)):= \bigcup_{m \in \Z} \{
\hbox{eigenvalues of } \hm(\theta) + \ep V_{\theta} \} \cap
{\mathcal S}_{0}(\theta).
$$
In other words, the set of  resonances of $H + \ep V$ is the union
with respect to $m \in \Z$ of the resonances of $\hm + \ep V$.
This definition is correct since the restriction of $H + \ep V$
onto ${\rm Ker}(L-m)$ is unitarily equivalent to $\hm + \ep V$.
Moreover using a standard deformation argument (see, for instance,
\cite[Chapter 16]{HiSi}), we can prove that these resonances
coincide with singularities of the function $z \mapsto \langle (H
+ \ep V -z)^{-1}\, f, \, f \rangle $, for $f$ in a dense
subset of $L^2(\re^3)$.
\begin{theorem}\label{thm1}
 Fix $m \in \Z$, $q > m_-$. Assume that:
\begin{itemize}
 \item $v_0 $   admits an analytic
extension in $S_{\theta_0}$ which satisfies \eqref{31}; \item
inequality \eqref{ser7} holds true, and $\hpa$ has a unique
discrete eigenvalue $\lambda$; \item $V$ is axisymmetric,
 and admits an analytic extension with respect to $x_3$ in $S_{\theta_0}$ which is $\hmo$-compact.
\end{itemize}
 Then for sufficiently small $|\ep|$, the
operator $\hm + \ep V$ has a  resonance $w_{q,m}(\ep)$ which obeys
the asymptotics
    \bel{vin4}
w_{q,m}(\ep) = 2bq + \lambda + \ep \langle V \Phi_{q,m},
\Phi_{q,m}\rangle -\ep^2 \, F_{q,m}(2bq+\lambda) +
O_{q,m,V}(\ep^3), \; \ep \to 0,
    \ee
    the eigenfunction $\Phi_{q,m}$ being defined in \eqref{vin55},
    and the quantity $F_{q,m}(2bq+\lambda)$ being defined in
    \eqref{ser4}.
\end{theorem}
\begin{proof}
 Fix $ \theta$ such that $ \theta_0> \hbox{Im } \theta \geq 0$ and assume that $z \in \C$ is in the resolvent set of the operator
$\hm(\theta) + \varkappa V_{\theta}$. Put
$$
R^{(m)}_{\ep,\theta}(z) : = (\hm(\theta) + \varkappa V_{\theta}-
z)^{-1}.
$$
 By the resolvent identity,  we have
\begin{equation}\label{resoeq}
R^{(m)}_{\ep,\theta}(z)= R^{(m)}_{0,\theta}(z) - \ep \,
R^{(m)}_{0,\theta}(z) V_\theta R^{(m)}_{0,\theta}(z) + \ep^2\,
R^{(m)}_{0,\theta}(z) V_\theta R^{(m)}_{0,\theta}(z)V_\theta
R^{(m)}_{0,\theta}(z) + O(\ep^3),
\end{equation}
as $\ep \to 0$, uniformly with respect to $z$ in a compact subset
of the resolvent
sets of $\hm(\theta) + \ep V$ and $\hm(\theta)$.\\
Now note that the simple embedded eigenvalue $2bq + \lambda$ of
$\hm$ is  a simple isolated eigenvalue of $\hm(\theta)$.
According to the Kato
perturbation theory (see \cite[Section VIII.2]{ka}), for
sufficiently small $\ep$ there exists a simple eigenvalue
$w_{q,m}(\ep)$ of $\hm(\theta) + \ep V_{\theta}$ such that
$\lim_{\ep \to 0} w_{q,m}(\ep) = w_{q,m}(0) = 2bq + \lambda$. For
$|\ep|$ sufficiently small, define the eigenprojector
\begin{equation}
\label{proj}
     {\mathcal P}_{\ep}(\theta) = {\mathcal
P}_{\ep, q,m}(\theta) : = \frac{-1}{2i\pi}\int_\Gamma
R^{(m)}_{\ep,\theta}(z) dz
\end{equation}
where $\Gamma$ is a small positively oriented circle centered at
$2bq + \lambda$. Evidently, for $u \in {\rm Ran}\,{\mathcal
P}_{\ep}(\theta)$ we have $(\hm(\theta) + \ep V_\theta) u =
w_{q,m}(\ep) u$; in particular, if $u \in {\rm Ran}\,{\mathcal
P}_{0}(\theta)$, then $\hm(\theta) u = (2bq + \lambda) u$. Since
$w_{q,m}(\ep)$ is a simple eigenvalue,  we have
\begin{equation}\label{sum}
 w_{q,m}(\ep) = \Tr \left( \frac{-1}{2i\pi}\int_\Gamma z  R^{(m)}_{\ep,\theta}(z) dz  \right)
\end{equation}
for $\Gamma$ and $\ep$ as above.
 Inserting (\ref{resoeq}) into \eqref{sum}, we get
$$
 w_{q,m}(\ep) =
 $$
 \begin{equation}\label{2.3bis}
 2bq + \lambda + \ep \Tr({\mathcal P}_{0}(\theta)\, V_\theta \,{\mathcal P}_{0}(\theta))
 -  \frac{\ep^2}{2i\pi} \Tr \left(\int_\Gamma z  R^{(m)}_{0,\theta}(z)\, V_\theta \,
 R^{(m)}_{0,\theta}(z)\, V_\theta \, R^{(m)}_{0,\theta}(z)\,  dz  \right) +
 O (\ep^3)
\end{equation}
as $\ep \to 0$. Next,  we have
    \bel{vin1}
\Tr({\mathcal P}_{0}(\theta)\, V_\theta \,{\mathcal
P}_{0}(\theta))  = \Tr({\mathcal P}_{q,m}\, V \,{\mathcal
P}_{q,m})  = \langle V \Phi_{q,m}, \Phi_{q,m}\rangle_{L^2(\re^+
\times \re;  \varrho d \varrho dx_3)},
 \ee
the orthogonal projection ${\mathcal P}_{q,m}$ being defined in
\eqref{vin50}. For $\theta \in \re$ the relation is obvious since
the operators ${\mathcal P}_{0,q,m}(\theta)\, V_\theta \,{\mathcal
P}_{0,q,m}(\theta)$ and ${\mathcal P}_{q,m}\, V \,{\mathcal
P}_{q,m}$ are unitarily equivalent. For general complex $\theta$
identity \eqref{vin1} follows from the fact the function $\theta
\mapsto \Tr({\mathcal P}_{0,q,m}(\theta)\, V_\theta \,{\mathcal
P}_{0,q,m}(\theta))$ is analytic.\\
Set $$ {\mathcal Q}_0(\theta) : = I - {\mathcal P}_0(\theta),
\quad \tihm(\theta) : = \hm(\theta) {\mathcal Q}_0(\theta).
$$
 By the cyclicity of the trace, we have
    \bel{vin2}
\Tr \left( \int_\Gamma z  R^{(m)}_{0,\theta}(z)\, V_\theta \,
 R^{(m)}_{0,\theta}(z)\, V_\theta \, R^{(m)}_{0,\theta}(z) dz \right) =
 T_1 + T_2 + T_3 + T_4
    \ee
    where
$$
T_1  : =    \int_\Gamma z   (2bq + \lambda -z)^{-3} \Tr
\left({\mathcal P}_{0}(\theta)\, V_\theta \,{\mathcal
P}_{0}(\theta) V_\theta \,{\mathcal P}_{0}(\theta)\right)dz,
$$
$$
T_2  : =   \int_\Gamma z   (2bq + \lambda -z)^{-2} \Tr
\left({\mathcal P}_{0}(\theta)\, V_\theta  \,(\tihm(\theta) -
z)^{-1}\, {\mathcal Q}_0(\theta)\, V_\theta \, {\mathcal
P}_{0}(\theta)\right) dz,
$$
$$
 T_3 : = \Tr \left(
\int_\Gamma z(\tihm(\theta) - z)^{-1} {\mathcal
Q}_0(\theta) V_\theta (\tihm(\theta) - z)^{-1} {\mathcal
Q}_0(\theta) V_\theta (\tihm(\theta) - z)^{-1} {\mathcal
Q}_0(\theta)  dz \right),
$$
$$
 T_4  : =
\int_\Gamma z \,(2bq + \lambda  -z)^{-1} \Tr\Big( {\mathcal
P}_0(\theta) \, V_\theta(\tihm(\theta) - z)^{-2}\, {\mathcal
Q}_0(\theta) \, V_\theta \, {\mathcal P}_0(\theta)\Big) dz.
$$
Since  $\int_\Gamma z   (2bq + \lambda -z)^{-3} dz=0$, and the
function $z \mapsto (\tihm(\theta) - z)^{-1}$ is analytic inside
$\Gamma$, we have
    \bel{vin10}
T_1=T_3=0.
    \ee
 Further, using the identity
$$
(2bq + \lambda -z)^{-2} {\mathcal P}_{0}(\theta)\, V_\theta
\,(\tihm(\theta) - z)^{-1}\, {\mathcal Q}_0(\theta) \, V_\theta \,
{\mathcal P}_{0}(\theta)  +
$$
$$
(2bq + \lambda -z)^{-1}  {\mathcal P}_{0}(\theta)\, V_\theta \,
(\tihm(\theta) - z)^{-2}\, {\mathcal Q}_0(\theta) \, V_\theta \,
{\mathcal P}_{0}(\theta) =
$$
$$
\frac{\partial}{\partial z} \left( (2bq + \lambda -z)^{-1}
{\mathcal P}_{0}(\theta)\, V_\theta \,(\tihm(\theta) - z)^{-1}\,
{\mathcal Q}_0(\theta) \, V_\theta \, {\mathcal P}_{0}(\theta)
 \right),
$$
integrating by parts, and applying the Cauchy theorem, we obtain
    \bel{vin11}
T_2+T_4= 2 i \pi \Tr \left( {\mathcal P}_{0}(\theta)\, V_\theta \,
(\tihm(\theta) - 2bq - \lambda)^{-1} \, (I-{\mathcal
P}_{0}(\theta))\, V_\theta \, {\mathcal P}_{0}(\theta) \right).
    \ee
Arguing as in the proof of \eqref{vin1}, we get
    \bel{vin12}
    \Tr\left({\mathcal P}_{0}(\theta)\, V_\theta \, (\tihm(\theta) - 2bq
- \lambda -i\delta)^{-1} \, {\mathcal Q}_{0}(\theta) \, V_\theta
\, {\mathcal P}_{0}(\theta) \right)
 = F_{q,m}(  2bq+\lambda + i \delta), \quad \delta > 0.
    \ee
For $\theta$ fixed such that $\hbox{Im}\, \theta>0$, the point
$2bq+ \lambda$ is not in the spectrum of $\tihm(\theta)$. Taking
the limit $\delta \downarrow 0$ in \eqref{vin12}, we find that
\eqref{vin11} implies
     \bel{vin3}
    T_2 + T_4 = 2i\pi \, F_{q,m}(  2bq+\lambda).
    \ee
   Putting together \eqref{2.3bis} -- \eqref{vin10}
    and \eqref{vin3}, we deduce    \eqref{vin4}.
\end{proof}
{\em Remarks}: (i) We will see in Section 4 that generically ${\rm
Im}\,
F_{q,m}(  2bq+\lambda)>0$ for all $m \in \Z$, and $q > m_-$.\\
(ii) Taking into account the above remark, we find that Theorem
\ref{thm1} implies that generically near $2bq + \lambda$, $q \geq
1$, there are infinitely many resonances of $H + \ep V$ with
sufficiently small $\ep$, namely the resonances of the operators
$\hm + \ep V$ with $m > -q$. \\

{\bf 3.2.} In this subsection we consider the dynamical aspect of
resonances. We prove the following proposition which will be
extended to non-analytic perturbations in Section 4.

\begin{pr}\label{dyndil}
Under the assumptions of Theorem \ref{thm1} there exists a
function $g \in C_0^{\infty}(\re; \re)$ such that  $g=1$ near $2bq + \lambda$, and
    \bel{x99}
\langle e^{-i(\hm + \ep V)t} g(\hm + \varkappa V) \Phi_{q,m},
\Phi_{q,m} \rangle = a(\varkappa) e^{-iw_{q,m}(\varkappa)t} +
b(\varkappa, t), \quad t \geq 0,
    \ee
with $a$ and $b$ satisfying the asymptotic estimates
$$
|a(\varkappa) - 1| = O(\varkappa^2),
$$
$$
|b(\varkappa, t)| = O(\varkappa^2 (1+t)^{-n}), \quad \forall n\in
{\mathbb Z}_+,
$$
as $\varkappa \to 0$ uniformly with respect to $t \geq 0$.
\end{pr}

 In order to prove the proposition, we will need the following
\begin{lemma} \label{lemexpp}
Set $$ {\mathcal Q}_\ep(\theta) : = I - {\mathcal P}_\ep(\theta),
\quad \tilde{R}^{(m)}_{\ep, \theta}(z) : = \Big( (\hm(\theta) +
\ep V_\theta) {\mathcal Q}_\ep(\theta)-z\Big)^{-1}{\mathcal
Q}_\ep(\theta),
$$
the projection ${\mathcal P}_\ep(\theta)$ being defined in
\eqref{proj}. Then for $|\ep|$ small enough, there exists a
finite-rank operator ${\mathcal F}_{\ep,\theta}^{(m)}$, uniformly
bounded with respect to $\ep$,  such that
    \bel{ppp1}
{\mathcal P}_{0}(\theta)={\mathcal P}_{\varkappa}(\theta)+ \ep
\Big(\tilde{R}^{(m)}_{\ep, \theta}(w_{q,m}(\varkappa)) V_\theta
{\mathcal P}_{\varkappa}(\theta) + {\mathcal
P}_{\varkappa}(\theta) V_\theta \tilde{R}^{(m)}_{\ep,
\theta}(w_{q,m}(\varkappa)) \Big) + \ep^2 {\mathcal
F}_{\ep,\theta}^{(m)}.
    \ee
\end{lemma}
\begin{proof}  By the resolvent identity,  we have
\begin{equation}\label{resoeq0}
R^{(m)}_{0,\theta}(\nu)= R^{(m)}_{\ep,\theta}(\nu) + \ep \,
R^{(m)}_{\ep,\theta}(\nu) V_\theta R^{(m)}_{\ep,\theta}(\nu) + \ep^2\,
R^{(m)}_{\ep,\theta}(\nu) V_\theta R^{(m)}_{\ep,\theta}(\nu)V_\theta
R^{(m)}_{0,\theta}(\nu).
\end{equation}
Moreover by definition of $\tilde{R}^{(m)}_{\ep, \theta}$ and of $\tihm:=\hm \, {\mathcal Q}_0$, we have:
\begin{equation}\label{decomp}
R^{(m)}_{\ep,\theta}(\nu)  = (w_{q,m}(\varkappa) -\nu)^{-1} {\mathcal P}_{\varkappa}(\theta) + \tilde{R}^{(m)}_{\ep, \theta}(\nu),
\end{equation}
$$
R^{(m)}_{0,\theta}(\nu)= (2bq + \lambda -\nu)^{-1} {\mathcal P}_{0}(\theta)  + (\tihm(\theta) -\nu)^{-1} {\mathcal Q}_{0}(\theta) ,
$$
where  $\nu \mapsto \tilde{R}^{(m)}_{\ep, \theta}(\nu)$ and  $\nu
\mapsto (\tihm(\theta) -\nu)^{-1} {\mathcal Q}_{0}(\theta) $ are
analytic near $2bq + \lambda$. Then, from the Cauchy formula, the
integration of (\ref{resoeq0}) on a small positively oriented
circle centered at $2bq + \lambda$, yields \eqref{ppp1} with
${\mathcal F}_{\ep,\theta}^{(m)}$ a linear combination of
finite-rank operators of the form $P_1 \, V_\theta \, P_2  \,
V_\theta \, P_3$, where $\{ {\mathcal P}_{0}(\theta) , \,
{\mathcal P}_{\ep}(\theta)\} \cap \{P_j, \, j=1,2,3\}\neq
\emptyset $, and
$$
P_j \in \{ {\mathcal P}_{0}(\theta) , \, {\mathcal P}_{\ep}(\theta), \,
\tilde{R}^{(m)}_{\ep, \theta}(\nu), \,  (\tihm(\theta) -\nu)^{-1}
{\mathcal Q}_{0}(\theta) , \, \hbox{ with } \nu =
w_{q,m}(\varkappa) \hbox{ or } \nu = 2bq + \lambda\}.
$$
Since these operators are uniformly bounded in $\ep$ with
$|\varkappa |$ small enough, ${\mathcal F}_{\ep,\theta}^{(m)}$ is
a finite-rank operator which also is uniformly bounded in
$\varkappa$ with $|\varkappa |$ small enough.
\end{proof}

{\it{Proof of Proposition \ref{dyndil}}}: Pick at first any $g \in
C_0^{\infty}(\re; \re)$ such that  $g=1$ near $2bq + \lambda$. We
have
$$
\langle e^{-i(\hm + \ep V)t}  g(\hm + \ep V) \Phi_{q,m}, \Phi_{q,m} \rangle =
{\rm Tr}\,(e^{-i(\hm + \ep V)t}  g(\hm + \ep V) {\mathcal P}_{q,m}) .
$$
By the Helffer-Sj\"ostrand formula,
\begin{equation}\label{hesjf}
  e^{-i(\hm + \ep V)t} \, g(\hm + \ep V)   \, {\mathcal P}_{q,m}  = \frac{1}{\pi} \int_{\re^2}      \frac{\partial
    \tilde{g}}{\partial \bar{z}}(z) \, e^{-izt}\,  (\hm  + \ep V- z)^{-1}   \, {\mathcal P}_{q,m}
    dx dy
 \end{equation}
where $z = x+iy$, $\bar{z} = x - iy$, $\tilde{g}$ is a compactly
supported, quasi-analytic extension of $g$,  and the convergence
of the integral is understood in the operator-norm sense (see e.g.
\cite[Chapter 8]{dsj}).\\
Consider the functions
$$
\sigma_\pm (z) := {\rm Tr}\,((\hm  + \ep V- z)^{-1}   \,
{\mathcal P}_{q,m} ) , \quad \pm {\rm Im}z >0.
$$
Following the arguments of the previous subsection, we find that
    \bel{ppp2}
\sigma_+(z)  = {\rm Tr}\,( {R}^{(m)}_{\ep, \theta}(z)  \,
{\mathcal P}_{0,q,m} (\theta)),  \quad {\rm Im}\,z
>0, \quad \theta_0> \hbox{Im}\,\theta >0.
    \ee
     Inserting \eqref{ppp1} into \eqref{ppp2}, and
using the cyclicity of the trace, and the elementary identities
$$
{\mathcal P}_{\varkappa}(\theta) \, {R}^{(m)}_{\ep, \theta}(z) \,
\tilde{R}^{(m)}_{\ep, \theta}(w_{q,m}(\varkappa))= 0 =
\tilde{R}^{(m)}_{\ep, \theta}(w_{q,m}(\varkappa)) \,
{\mathcal P}_{\varkappa}(\theta)  \, {R}^{(m)}_{\ep, \theta}(z),
$$
we get
$$
\sigma_+(z)= {\rm Tr}\,( {R}^{(m)}_{\ep, \theta}(z)  \, {\mathcal
P}_{\ep} (\theta))+ \ep^2 {\rm Tr}\,( {R}^{(m)}_{\ep, \theta}(z)
\, {\mathcal F}_{\ep,\theta}^{(m)}).
$$
Applying (\ref{decomp}), we obtain
    \bel{pp3a}
\sigma_+(z) =  \Big( 1 + \ep^2 r(\ep)\Big) (w_{q,m}(\varkappa)
-z)^{-1} + \ep^2 G_+(\ep,z),
    \ee
where $r(\ep) : = {\rm Tr}\,( {\mathcal P}_{\ep}(\theta) \,
{\mathcal F}_{\ep,\theta}^{(m)})$, and $G_+(\ep,z) $ is analytic
near $2bq+\lambda$ and uniformly bounded with respect to $|\ep|$
small enough. Similarly,
    \bel{pp3b}
\sigma_-(z)=  \left( 1 + \ep^2\overline{r(\ep)}\right)
(\overline{w_{q,m}(\varkappa)} -z)^{-1} + \ep^2 G_-(\ep,z),
    \ee
where $G_-(\ep,z) $ is analytic near $2bq+\lambda$ and uniformly
bounded with respect to $|\ep|$ small enough.
 Now, assume that the support of $g$ is such that we can choose  $\tilde{g}$
 supported on a neighborhood of $2bq+\lambda$ where
 the functions
  $z \mapsto G_\pm (\ep,z)$ are holomorphic. Combining (\ref{hesjf}) with the
  Green formula, we get
    \bel{ppp3}
 {\rm Tr}\,(e^{-i(\hm + \ep V)t} \, g(\hm + \ep V)   \, {\mathcal P}_{q,m} ) =
 \frac{1}{2i\pi} \int_{\re} g(\mu)\, e^{-i\mu t}\,(  \sigma_+(\mu) -\sigma_-(\mu))
 d\mu.
    \ee
    Making use of \eqref{pp3a} -- \eqref{pp3b}, we
    get
    $$
\frac{1}{2i\pi} \int_{\re} g(\mu)\, e^{-i\mu t}\,(  \sigma_+(\mu)
-\sigma_-(\mu))  d \mu =
 \frac{\ep^2}{2i\pi} \int_{\re} g(\mu) \,e^{-i\mu t}\,
 (G_+(\ep,\mu) - G_-(\ep,\mu))  d \mu
 $$
$$
 + \frac{1+ \ep^2 r(\ep)}{2i\pi} \int_{\re}      g(\mu)   \, e^{-i\mu t}\,
 ({w_{q,m}(\varkappa)-\mu})^{-1} d \mu
 $$
 $$
 - \frac{1+ \ep^2 \overline{r(\ep)}}{2i\pi} \int_{\re}      g(\mu)
\, e^{-i\mu t}\,
 (\overline{w_{q,m}(\varkappa)}-\mu)^{-1} d \mu.
 $$
 Pick $\varepsilon > 0$ so small that $g(\mu) = 1$
 for $\mu \in [2bq + \lambda - 2\varepsilon, 2bq + \lambda +
 2\varepsilon]$. Set
 $$
 C_{\varepsilon} : = (-\infty, 2bq + \lambda - \varepsilon] \cup
 \{2bq + \lambda  + \varepsilon e^{it}, \, t \in [-\pi,0]\}
 \cup [2bq + \lambda + \varepsilon, + \infty),
 $$
 $$
 g(\mu) : = 1, \quad \mu \in C_{\varepsilon} \setminus \re.
 $$
 Taking into account \eqref{ppp3}, bearing in mind that ${\rm Im}\,w_{q,m}(\ep) < 0$,
 and applying the Cauchy theorem, we easily find that
    \bel{pp50}
    {\rm Tr}\,(e^{-i(\hm + \ep V)t} \, g(\hm + \ep V)   \, {\mathcal
    P}_{q,m}) = (1 + \ep^2 r(\ep))e ^{-i w_{q,m}(\ep)t} + \ep^2
    \sum_{j=1,2,3} I_j(t; \ep)
    \ee
    where
    $$
    I_1(t; \ep) : =   \frac{1}{2i\pi} \int_{\re} g(\mu) \,e^{-i\mu t}\,
 (G_+(\ep,\mu) - G_-(\ep,\mu))  d \mu,
 $$
 $$
 I_2(t; \ep) : =  \frac{1}{2i\pi}
\int_{C_{\varepsilon}}g(\mu) \,e^{-i\mu t} (r(\ep)(w_{q,m}(\ep) -
\mu)^{-1} - \overline{r(\ep)} (\overline{w_{q,m}(\ep)} -
\mu)^{-1})
  d \mu,
  $$
 $$
I_3(t; \ep) : =  - \frac{{\rm Im}\,w_{q,m}(\ep)}{\ep^2\pi}
\int_{C_{\varepsilon}} g(\mu) \,e^{-i\mu t} (w_{q,m}(\ep) -
\mu)^{-1} (\overline{w_{q,m}(\ep)} - \mu)^{-1} \,
  d \mu.
  $$
Integrating by parts, we find that
    \bel{pp51}
    |I_j(t; \ep)| = O ((1+t)^{-n}), \quad t>0, \quad
    j=1,2,3, \quad \forall n \in {\mathbb Z}_+,
    \ee
    uniformly with respect to $\ep$,
    provided that $|\ep|$ is small
    enough; in the estimate of $I_3(t; \ep)$ we have taken into account
    that
    by Theorem \ref{thm1} we have
     $|{\rm Im}(w_{q,m}(\varkappa))| = O(\ep^2)$. Putting together
     \eqref{pp50} and \eqref{pp51}, we get
\eqref{x99}.

\section{Mourre estimates and dynamical resonances}
\setcounter{equation}{0}

In this section we obtain Mourre estimates for the operator $\hm$
and apply them combined with a recent result of Cattaneo, Graf,
and Hunziker (see \cite{cgh}) in order to investigate the dynamics
of the resonance states of the operator $\hm$ without analytic
assumptions.\\

{\bf 4.1.} Let $v_0 : \re \to \re$. Set
    \bel{ser3}
v_j(x_3) : = x_3^j v_0^{(j)}(x_3), \quad j \in {\mathbb Z}_+,
    \ee
provided that the corresponding derivative $v_0^{(j)}$ of $v_0$
is well-defined.\\
Let
$$
{\mathcal A} : = -\frac{i}{2}\left(x_3 \frac{d }{d x_3} + \frac{d
}{d x_3} x_3\right)
$$
be the self-adjoint operator defined initially on
$C_0^{\infty}(\re)$ and then closed in
  $L^2(\re)$. Set $A : = \tilde{I}_{\perp} \otimes {\mathcal A}$.
  Let $T$ be an operator self-adjoint in $L^2(\re_+ \times \re;
  \varrho d\varrho dx_3)$ such that $e^{isA} D(T) \subseteq D(T)$,
  $s \in \re$. Define the commutator $[T,iA]$ in the sense of
  \cite{JMP} and \cite{cgh}, and set
  ${\rm ad}_A^{(1)}(T) : = -i [T,
  iA]$. Define recursively
$$
{\rm ad}_A^{(k+1)}(T) = -i [{\rm ad}_A^{(k)}(T),iA], \quad k \geq
1,
$$
provided that the higher order commutators are well-defined.
Evidently, for each $m \in {\mathbb Z}$ we have
    \bel{ser5}
 i^k {\rm ad}_A^{(k)}(\hm) = 2^k \gaden + \sum_{j=1}^k c_{k,j} v_j, \quad
 k \in {\mathbb Z}_+ ,
    \ee
 with some constants $c_{k,j}$ independent of $m$; in particular, $c_{k,k} = (-1)^k$.
 Therefore the $\hopa$-boundedness of
 the multipliers $v_j$, $j=1,\ldots, k$, guarantees the
 $\hm$-boundedness of all the operators ${\rm ad}_A^{(j)}(\hm)$ , $j=1,\ldots,
 k$. \\
  Let $J \subset \re$ be
a Borel set, and $T$ be a self-adjoint operator. Denote by
${\mathbb P}_J(T)$  the spectral projection of the operator $T$
associated with $J$.
\begin{pr} \label{pr1}
Fix $m \in {\mathbb Z}$. Let $\lambda \in (-2b,0)$, $q \in
{\mathbb Z}$, $q > m_-$. Put
    \bel{ser8}
J = (2bq + \lambda - \delta, 2bq + \lambda +\delta)
    \ee
where $\delta > 0$, $\delta < -\lambda/2$, and $\delta < (2b +
\lambda)/2$. Assume that the operators $v_j (\hopa +1)^{-1}$,
$j=0,1$, are compact. Then there exist a positive constant $C
> 0$ and a compact operator $K$ such that
    \bel{1}
    {\mathbb P}_J(\hm) [\hm,iA] {\mathbb P}_J(\hm) \geq C {\mathbb P}_J(\hm) + K.
    \ee
\end{pr}
\begin{proof}
Let $\chi \in C_0^{\infty}(\re;\re)$ be such that ${\rm
supp}\;\chi = [2bq+ \lambda - 2\delta, 2bq + \lambda + 2\delta]$,
$\chi (t) \in [0,1]$, $\forall t \in \re$, $\chi (t) = 1$,
$\forall t \in J$. In order to prove \eqref{1}, it suffices to
show that
    \bel{2}
    \ph [\hm,iA] \ph \geq C \ph^2 + \tilde{K}
    \ee
with a compact operator $\tilde{K}$. Indeed, if inequality
\eqref{2} holds true, we can multiply it from the left and from
the right by ${\mathbb P}_J(\hm)$ obtaining thus \eqref{1} with $K
= {\mathbb P}_J(\hm) \tilde{K} {\mathbb P}_J(\hm)$.\\
Next \eqref{ser5} yields
$$
[\hm,iA] = 2 \gaden - v_1.
$$
Therefore,
$$
\ph [\hm,iA] \ph = 2 \ph \left( \gaden \right) \ph - \ph v_1 \ph =
$$
    \bel{3}
    2 \pho \left( \gaden \right) \pho + 2K_1 - K_2
    \ee
where
$$
K_1 : = \ph \left( \gaden \right) \ph - \pho \left( \gaden \right)
\pho,
$$
$$
K_2 : = \ph v_1 \ph.
$$
Since the operator $v_1 (\hmo + 1)^{-1}$ is compact, the operators
$v_1 (\hm + 1)^{-1}$, $v_1 \ph$, and $K_2$ are compact as well.
Let us show that $K_1$ is also compact. We have
    $$
    K_1  = (\ph - \pho) \left( \gaden \right) \ph +
    $$
    \bel{5a}
    \pho \left(
\gaden \right) (\ph - \pho ).
    \ee
    By the Helffer-Sj\"ostrand formula and the resolvent identity,
    $$
    \ph - \pho = - \frac{1}{\pi} \int_{\re^2} \frac{\partial
    \tilde{\chi}}{\partial \bar{z}}(z) (\hm - z)^{-1} v_0 (\hmo - z)^{-1}
    dx dy.
    $$
 Since the support of $\tilde{\chi}$ is
compact in $\re^2$, and the operator $\frac{\partial
    \tilde{\chi}}{\partial \bar{z}} (\hm - z)^{-1} v_0 (\hmo - z)^{-1}$ is compact for
every $(x,y) \in \re^2$ with $y \neq 0$, and is uniformly
norm-bounded for every $(x,y) \in \re^2$, we find that the
operator $\ph - \pho$ is compact. On the other hand, it is easy to
see that the operators
$$
\left( \gaden \right) \ph = \left( \gaden \right)
(\hm+1)^{-1}(\hm+1) \ph =
$$
$$
\left( \gaden \right) (\hmo+1)^{-1}(I - v_0(\hm+1)^{-1})(\hm+1)
\ph
$$
and
$$
\pho \left( \gaden \right) = \pho (\hmo + 1)(\hmo+1)^{-1}\left(
\gaden \right)
$$
are bounded. Taking into account \eqref{5a}, and bearing in mind
the compactness of the operator $\ph - \pho$, and the boundedness
of the operators $\left( \gaden \right) \ph$ and $\pho \left(
\gaden \right)$, we conclude that the operator $K_1$ is compact.
\\
Further, since $\delta < -\lambda/2$, and hence $2bj > 2bq +
\lambda + 2\delta$ for all $j \geq q$, we have
$$
\chi(\hopa + 2bj) = 0, \quad j \geq q.
$$
Therefore,
$$
\pho = \sum_{j=m_-}^{\infty} \tilde{p}_{j,m}  \otimes \chi(\hopa +
2bj) = \sum_{j=m_-}^{q-1} \tilde{p}_{j,m} \otimes \chi(\hopa +
2bj),
$$
and
    \bel{4}
    \pho  \left( \gaden \right) \pho = \sum_{j=m_-}^{q-1}
    \tilde{p}_{j,m} \otimes \left(\chi(\hopa +
    2bj)^2 \hopa\right).
    \ee
By the spectral theorem,
$$
\sum_{j=m_-}^{q-1} \tilde{p}_{j,m} \otimes \left(\chi(\hopa +
2bj)^2\hopa \right)   \geq \sum_{j=m_-}^{q-1}(2b(q-j) + \lambda -
2\delta) \tilde{p}_{j,m} \otimes \chi(\hopa + 2bj)^2 \geq
$$
    \bel{5}
    (2b + \lambda - 2\delta)\sum_{j=m_-}^{q-1} \tilde{p}_{j,m} \otimes \chi(\hopa +
    2bj)^2 = C_1 \chi(\hmo)^2
    \ee
with $C_1 : = 2b + \lambda - 2\delta > 0$. Combining \eqref{3},
\eqref{4}, and \eqref{5}, we get
    \bel{6}
    \ph [\hm, iA] \ph \geq 2C_1 \ph^2 + 2C_1 K_3 + 2K_1 - K_2
    \ee
where $K_3 : = \pho^2 - \ph^2$ is a compact operator by the
Helffer-Sj\"ostrand formula. Now we find that \eqref{6} is
equivalent to \eqref{2} with $C = 2C_1$ and $\tilde{K} = 2C_1 K_3
+ 2K_1 - K_2$.
\end{proof}
{\em Remark}: Mourre estimates for various magnetic quantum
Hamiltonians can be found in \cite[Chapter 3]{gl}.\\

 {\bf 4.2.} By analogy with \eqref{ser3} set
$$
V_j(\varrho, x_3) = x_3^j \frac{\partial^j V(\varrho,
x_3)}{\partial x_3^j}, \quad j \in {\mathbb Z}_+.
$$
We have
$$
i^k {\rm ad}_A^{(k)}(V) = \sum_{j=1}^k c_{k,j} V_j
$$
with the same constants $c_{k,j}$ as in \eqref{ser5}. \\
We will say that the condition ${\mathcal O}_{\nu}$, $\nu \in
{\mathbb Z}_+$, holds true if the multipliers by $v_j$, $j=0,1$,
are $\hopa$-compact,  and the multipliers by $v_j$, $j \leq \nu$,
are $\hopa$-bounded. Also, for a fixed $m \in {\mathbb Z}$ we will
say that the condition
 ${\mathcal C}_{\nu, m}$, $\nu \in {\mathbb Z}_+$, holds true if
the condition ${\mathcal O}_{\nu}$ is valid, the multiplier by $V$
is  $\hmo$-bounded with zero relative bound, and the multipliers
by $V_j$, $j = 1, \ldots,
\nu$, are $\hmo$-bounded.\\
By Proposition \ref{pr1} and \cite[Lemma 3.1]{cgh}, the validity
of condition ${\mathcal C}_{\nu, m}$ with $\nu \geq 5$ and a given
$m \in {\mathbb Z}$ guarantees the existence of a finite limit
$F_{q,m}(2bq + \lambda)$ with $q > m_-$ in \eqref{ser4}, provided
that \eqref{ser7}
holds true, and $\lambda$ is a discrete eigenvalue of $\hopa + v_0$.\\
Combining the results of Proposition \ref{pr1} and \cite[Theorem
1.2]{cgh}, we obtain the following
\begin{theorem} \label{th1}
Fix $m \in {\mathbb Z}$, $n \in {\mathbb Z}_+$. Assume that:
\begin{itemize}
\item the condition ${\mathcal C}_{\nu, m}$ holds with $\nu \geq
n+5$; \item inequality \eqref{ser7} holds true, and $\lambda$ is a
discrete eigenvalue of $\hpa$; \item inequality \eqref{ser4a}
holds true, and hence the Fermi Golden Rule ${\mathcal
F}_{q,m,\lambda}$ is valid.
\end{itemize}
Then there exists a
function $g \in C_0^{\infty}(\re; \re)$ such that ${\rm supp}\,g =
\overline{J}$ (see \eqref{ser8}), $g=1$ near $2bq + \lambda$, and
    \bel{x9}
\langle e^{-i(\hm + \ep V)t} g(\hm + \varkappa V) \Phi_{q,m},
\Phi_{q,m} \rangle = a(\varkappa) e^{-i\lambda_{q,m}(\varkappa)t}
+ b(\varkappa, t), \quad t \geq 0,
    \ee
where
    \bel{x10}
\lambda_{q,m}(\varkappa) = 2bq +\lambda + \varkappa \langle V
\Phi_{q,m}, \Phi_{q,m}\rangle - \varkappa^2 F_{q,m}(2bq + \lambda)
+ o_{q,m,V}(\varkappa^2), \quad \varkappa \to 0.
    \ee
In particular, we have ${\rm Im}\,\lambda_{q,m}(\varkappa) < 0$ for
$|\varkappa|$ small enough. Moreover, $a$ and $b$ satisfy the
asymptotic estimates
$$
|a(\varkappa) - 1| = O(\varkappa^2),
$$
$$
|b(\varkappa, t)| = O(\varkappa^2 |\ln{|\varkappa|}| (1+t)^{-n}),
$$
$$
|b(\varkappa, t)| = O(\varkappa^2 (1+t)^{-n+1}),
$$
as $\varkappa \to 0$ uniformly with respect to $t \geq 0$.
\end{theorem}
We will say that the condition
 ${\mathcal C}_{\nu}$, $\nu \in {\mathbb Z}_+$, holds true if
the condition ${\mathcal O}_{\nu}$ is valid, the multiplier by $V$
is  $H_0$-bounded with zero relative bound, and the multipliers by
$V_j$, $j = 1, \ldots,
\nu$, are $H_0$-bounded.\\
For $m \in {\mathbb Z}$ and $q \geq m_-$ denote by
$\tilde{\Phi}_{q,m}: \rt \to {\mathbb C}$ the function written in
cylindrical coordinates $(\varrho, \phi, x_3)$ as
$\tilde{\Phi}_{q,m}(\varrho, \phi, x_3) =(2\pi)^{-\frac12}
e^{im\phi}\Phi_{q,m}(\varrho, x_3)$.

\begin{follow} \label{f31}
Fix  $n \in {\mathbb Z}_+$. Assume that:
\begin{itemize}
\item
 the condition ${\mathcal
C}_{\nu}$ holds with $\nu \geq n+5$; \item inequality \eqref{ser7}
is fulfilled, and $\lambda$ is a discrete eigenvalue of $\hopa +
v_0$; \item for each $m \in {\mathbb Z}$, $q > m_-$, inequality
\eqref{ser4a} holds true, and hence the Fermi Golden Rule
${\mathcal F}_{q,m,\lambda}$ is valid.
\end{itemize}
 Then   for every fixed $q \in {\mathbb Z}_+$, and each $m \in
\{-q+1, \ldots, 0\} \cup {\mathbb N}$ with ${\mathbb N} : =
\{1,2,\ldots\}$, we have
    $$
\langle e^{-i(H + \varkappa V)t} g(H + \varkappa V)
\tilde{\Phi}_{q,m}, \tilde{\Phi}_{q,m} \rangle_{L^2(\rt)} =
a(\varkappa) e^{-i \lambda_{q,m}(\ep) t} + b(\varkappa, t), \quad
t \geq 0,
    $$
    where $g$, $\lambda_{q,m}(\ep)$,  $a$, and $b$ are the same as
    in Theorem \ref{th1}.
\end{follow}
{\em Remarks}: (i) If $q \geq 1$, then Corollary \ref{f31} tells
us that typically the eigenvalue $2bq + \lambda$ of the operator
$H$, which has an infinite multiplicity, generates under the
perturbation $\varkappa V$ infinitely many resonances with
non-zero imaginary part. Note however that $2bq+\lambda$ is a
discrete simple eigenvalue of the operator $H^{(-q)}$, and
therefore the operator $H^{(-q)} + \varkappa V$ has a simple
discrete eigenvalue provided that $|\varkappa|$ is small enough.
Generically, this eigenvalue is an embedded eigenvalue for the
operator $H + \varkappa V$.\\
(ii) If $q=0$, then $\lambda$ is an isolated eigenvalue of
infinite multiplicity for $H$. By Theorem \ref{t31} below, in this
case there exists an infinite series of discrete eigenvalues of
the operator $H + V$ which accumulate at $\lambda$, provided that
the perturbation $V$ has a definite sign.

\section{Sufficient conditions for the validity of the Fermi
Golden Rule}
    \setcounter{equation}{0}
In this section we describe certain classes of perturbations
 $V$ compatible with the hypotheses of Theorems \ref{thm1} - \ref{th1}, for which
 the Fermi Golden rule
 ${\mathcal F}_{q,m,\lambda}$ is valid for
every $m \in {\mathbb Z}$ and $q > m_-$. The results included are
of two different types. Those of Subsection 5.1 are  less general
but they offer a constructive approximation of $V$ by potentials
for which the Fermi Golden Rule holds. On the other hand, the
results of Subsection 5.2 are more general,
but they are more abstract and less constructive.\\

{\bf 5.1.} Assume that $v_0 \in C^{\infty}(\re)$ satisfies the
estimates
    \bel{ppp100}
    |v_0^{(j)}(x)| = O_j \left(\langle x \rangle^{-m_0 - j}\right),
    \quad x \in \re, \quad j \in {\mathbb Z}_+, \quad m_0 > 1.
    \ee
Then condition ${\mathcal O}_{\nu}$ is valid for every $\nu \in
{\mathbb Z}_+$.  Moreover, in this case the eigenfunction $\psi$
(see \eqref{ser2}) is in the Schwartz class ${\mathcal S}(\re)$,
while the Jost solutions $y_j(\cdot;k)$, $j=1,2$, belong to
$C^{\infty}(\re) \cap
L^{\infty}(\re)$.\\
Suppose that \eqref{ser7} holds true, and the discrete spectrum of
the operator $\hpa$ consists of a unique eigenvalue $\lambda$. Fix
$m \in {\mathbb Z}$, and $q \in {\mathbb Z}_+$ such that $q
> m_-$. Then it is easy to check that we have
$$
{\rm Im}\,F_{q,m}(2bq + \lambda) =
$$
    \bel{fgr2}
\pi \sum_{l=1,2} \sum_{j=m_-}^{q-1}\left|\int_0^{\infty}
\int_{\re} \varphi_{j,m}(\varrho) \varphi_{q,m}(\varrho) \psi(x_3)
\Psi_l(x_3; 2b(q-j) + \lambda) V(\varrho, x_3) dx_3 \varrho
d\varrho\right|^2.
    \ee
In what follows we denote by $\lre$ the set of real functions $W
\in L^2(\re_+; \varrho d\varrho)$.
    \begin{lemma} \label{l31}
The set of functions $W \in \lre$ for which
    \bel{pp1}
\int_0^{\infty}  \varphi_{q-1,m}(\varrho) \varphi_{q,m}(\varrho)
 W(\varrho) \varrho d\varrho \neq 0
    \ee
 for every $m \in {\mathbb Z}$, $q > m_-$, is dense in $\lre$.
    \end{lemma}
    \begin{proof}
 Since the Laguerre polynomials $L_q^{(0)}$, $q \in {\mathbb
 Z}_+$, (see \eqref{x3}) form an orthogonal basis in $L^{2}(\re_+; e^{-s}ds)$,
 the set of polynomials is dense in $L^{2}(\re_+; e^{-s}ds)$. Pick $W \in \lre$. Set $w(s)
 : = W(\sqrt{2s/b})e^{s/2}$, $s>0$. Evidently, $w = \overline{w} \in
L^{2}(\re_+;
 e^{-s}ds)$. Pick $\varepsilon > 0$ and find a non-zero polynomial
 ${\mathcal P}$ with real coefficients such that
 $$
 \int_0^{\infty} e^{-s} \left({\mathcal P}(s) - w(s)\right)^2 ds < \frac{b
 \varepsilon^2}{4}.
 $$
 Note that the coefficients of ${\mathcal P}$ could be chosen real
 since the coefficients of the Laguerre polynomials are real.
 Changing the variable $s = b\varrho^2/2$, we get
    \bel{x1}
    \int_0^{\infty} \left({\mathcal P}(b\varrho^2/2)
    e^{-b\varrho^2/4} - W(\varrho)\right)^2 \varrho d\varrho <
    \frac{\varepsilon^2}{4}.
    \ee
    Now set ${\mathcal W}_{\alpha}(\varrho) : = {\mathcal P}(b\varrho^2/2)
    e^{-\alpha b\varrho^2/2}$, $\varrho \in \re_+$, $\alpha \in (0,\infty)$, where the
    real polynomial ${\mathcal P}$ is fixed and satisfies
    \eqref{x1}. We will show that the set
    \bel{pp2}
{\mathbb A} : = \left\{\alpha \in (0,\infty) \, | \,
\int_0^{\infty} \varphi_{q-1,m}(\varrho) \varphi_{q,m}(\varrho)
 {\mathcal W}_{\alpha}(\varrho) \varrho d\varrho \neq 0, \;
 \forall m \in {\mathbb Z}, \; \forall q > m_- \right\}
    \ee
 is dense in $(0,\infty)$. Actually, for fixed $m \in {\mathbb Z}$ and
 $q > m_-$, we have
 $$
\int_0^{\infty}  \varphi_{q-1,m}(\varrho) \varphi_{q,m}(\varrho)
 {\mathcal W}_{\alpha}(\varrho) \varrho d\varrho =
 \Pi_{q,m}(\gamma(\alpha))
 $$
 where $\Pi_{q,m}$ is a real polynomial of degree $2q + m + 1 +
 {\rm deg}\, {\mathcal P}$, and $\gamma(\alpha) : =
 (1+\alpha)^{-1}$. Note that $\gamma : (0,\infty) \to (0,1)$ is a
 bijection. Denote by ${\mathcal N}_{q,m}$ the set of the zeros of
 $\Pi_{q,m}$ lying on the interval $(0,1)$. Set
 $$
 {\mathcal N} : = \bigcup_{m \in {\mathbb Z}} \bigcup_{q =
 m_- + 1}^{\infty} {\mathcal N}_{q,m}.
 $$
 Evidently, the sets ${\mathcal N}$ and ${\gamma}^{-1}({\mathcal
 N})$ are countable, and ${\mathbb A} = (0,\infty)\setminus {\gamma}^{-1}({\mathcal
 N})$. Therefore, ${\mathbb A}$ is dense in $(0,\infty)$.
 Now, pick $\alpha_0 \in {\mathbb A}$ so close to $1/2$ that
    \bel{x5}
\int_0^{\infty} {\mathcal
P}(b\varrho^2/2)^2\left(e^{-b\varrho^2/4} - e^{-\alpha_0
b\varrho^2/2}\right)^2 \varrho d\varrho < \frac{\varepsilon^2}{4}.
    \ee
    Assembling \eqref{x1} and \eqref{x5}, we obtain
    \bel{pp3}
    \|W - {\mathcal W}_{\alpha_0}\|_{L^2(\re_+;\varrho d\varrho)} < \varepsilon.
    \ee
\end{proof}
Denote by $K(\re)$ the class of real-valued continuous functions
$u : [0,\infty) \to \re$ such that $\lim_{s \to \infty} u(s) = 0$.
Set $\|u\|_{K(\re)} : = \max_{s \in [0,\infty)} |u(s)|$.
\begin{lemma} \label{l31a}
The set of functions $W \in K(\re)$ for which
    \eqref{pp1} holds true
for every $m \in {\mathbb Z}$, $q > m_-$, is dense in $K(\re)$.
    \end{lemma}
    \begin{proof}
By the Stone-Weierstrass theorem for locally compact spaces, we
find that the set of functions $e^{-\alpha s} {\mathcal P}(s)$,
$s>0$ where $\alpha \in (0,\infty)$, and ${\mathcal P}$ is a
polynomial,
is dense in $K(\re)$.\\
Let $W \in K(\re)$; then we have $u \in K(\re)$ where $u(s) : =
W(\sqrt{2s/b})$, $s>0$. Pick $\varepsilon > 0$ and find $\alpha
\in (0,\infty)$ and a polynomial ${\mathcal P}$ such that $\|W -
{\mathcal W}_{\alpha}\|_{K(\re)} < \varepsilon/2$ where, as in the
proof of Lemma \ref{l31}, ${\mathcal W}_{\alpha}(\varrho) =
e^{-\alpha b \varrho^2/2} {\mathcal P}( b \varrho^2/2)$. Next pick
$\alpha_0 \in {\mathbb A}$ (see \eqref{pp2}) such that
$\|{\mathcal W}_{\alpha} - {\mathcal W}_{\alpha_0}\|_{K(\re)} \leq
\varepsilon/2$. Therefore, similarly to  \eqref{pp3} we have $\|W
- {\mathcal W}_{\alpha_0}\|_{K(\re)} \leq \varepsilon$.
 \end{proof}
Fix $\nu \in {\mathbb Z}_+$. We will write $V \in {\mathcal
D}_{\nu}$ if
$$
\|V\|_{{\mathcal D}_{\nu}}^2 : = \sum_{j=0}^{\nu} \int_0^{\infty}
\int_{\re} \left(x_3^j \frac{\partial^j V(\varrho, x_3)}{\partial
x_3^j}\right)^2 dx_3 \varrho d\varrho < \infty.
$$
Note that if ${\mathcal O}_{\nu}$ holds, and $V \in {\mathcal
D}_{\nu}$, then ${\mathcal C}_{\nu}$ is valid.
    \begin{theorem} \label{t32}
Assume that:
    \begin{itemize}
    \item $v_0 \in C^{\infty}(\re)$ satisfies \eqref{ppp100};
    \item inequality \eqref{ser7} holds true;
    \item we have $\sigma_{\rm disc}(\hpa) = \{\lambda\}$.
    \end{itemize}
  Fix $\nu \in
{\mathbb Z}_+$. Then the set of real perturbations $V : \re_+
\times \re \to \re$ for which the Fermi Golden Rule ${\mathcal
F}_{q,m,\lambda}$ is valid for each $m \in {\mathbb Z}$ and $q >
m_-$, is dense in ${\mathcal D}_{\nu}$.
    \end{theorem}
    \begin{proof}
    We will prove that the set of perturbations $V$ for which the
    integral
    \bel{pp60}
    I_{q,m, \lambda}(V) : = {\rm Re}\,\int_0^{\infty} \int_{\re} \varphi_{q-1,m}(\varrho) \varphi_{q,m}(\varrho) \psi(x_3)
\Psi_1(x_3; 2b + \lambda)  V(\varrho, x_3) dx_3 \varrho d\varrho
    \ee
    does not vanish for each $m \in {\mathbb Z}$ and $q > m_-$, is dense in
${\mathcal D}_{\nu}$. By \eqref{fgr2} this will imply the claim of
the theorem. Set
$$
\omega(x) : = \psi(x) {\rm Re}\,\Psi_1(x; 2b + \lambda), \quad x
\in \re.
$$
Note that $0 \neq \omega = \overline{\omega} \in {\mathcal
S}(\re)$. Set
$$
V_{\perp}(\varrho) = \int_{\re}  \omega(x_3) V(\varrho, x_3) dx_3,
\quad \varrho \in \re_+.
$$
Evidently, $V_{\perp} \in \lre$. Fix $\varepsilon > 0$. Applying
Lemma \ref{l31}, we find $\tilde{V}_{\perp} \in \lre$ such that
    \bel{x6}
    \int_0^{\infty}  \varphi_{q-1,m}(\varrho) \varphi_{q,m}(\varrho)
 \tilde{V}_{\perp}(\varrho) \varrho d\varrho \neq 0
    \ee
 for every $m \in {\mathbb Z}$, $q > m_-$, and
    \bel{x7}
    \|V_{\perp} - \tilde{V}_{\perp}\|_{L^2(\re_+;\varrho d\varrho)} < \varepsilon.
    \ee
    Set
    \bel{x8}
    \tilde{V}(\varrho, x_3) : = \frac{\tilde{V}_{\perp}(\varrho)
    \omega(x_3)}{\|\omega\|_{L^2(\re)}^2} + V(\varrho, x_3) - \frac{V_{\perp}(\varrho)
    \omega(x_3)}{\|\omega\|_{L^2(\re)}^2}, \quad \varrho \in \re_+,
    \quad x_3 \in \re.
    \ee
    We have
    $$
    I_{q,m,\lambda}(\tilde{V}) = \int_0^{\infty}  \varphi_{q-1,m}(\varrho) \varphi_{q,m}(\varrho)
 \tilde{V}_{\perp}(\varrho) \varrho d\varrho \neq 0
 $$
 for every $m \in {\mathbb Z}$, $q > m_-$. On the other hand,
 \eqref{x7} and \eqref{x8} imply
 $$
 \|V-\tilde{V}\|_{{\mathcal D}_{\nu}}^2 \leq
 \varepsilon^2
 \frac{\sum_{j=0}^{\nu} \int_{\re} x^{2j} \omega^{(j)}(x)^2 dx}{\|\omega\|^2_{L^2(\re)}}.
 $$
\end{proof}
Fix again $\nu \in {\mathbb Z}_+$. We will write $V \in {\mathcal
E}_{\nu}$ if $V : \re_+ \to \re$ is continuous and  tends to zero
at infinity, and the functions $x_3^j \frac{\partial^j V(\varrho,
x_3)}{\partial x_3^j}$, $j=1,\ldots, \nu$, are bounded. If
${\mathcal O}_{\nu}$ holds, and $V \in {\mathcal E}_{\nu}$, then
${\mathcal C}_{\nu}$ holds true. For $V \in {\mathcal E}_{\nu}$
define the norm
$$
\|V\|_{{\mathcal E}_{\nu}} : = \sum_{j=0}^{\nu}
\sup_{(\varrho,x_3) \in \re_+ \times \re}  \left|x_3^j
\frac{\partial^j V(\varrho, x_3)}{\partial x_3^j}\right|.
$$
Arguing as in the proof of Theorem \ref{t32}, from Lemma \ref{l31a} we obtain the
following
  \begin{theorem} \label{t32aa}
Assume that $v_0$ satisfies the hypotheses of Theorem \ref{t32}.
Fix $\nu \in {\mathbb Z}_+$. Then the set of perturbations $V :
\re_+ \times \re \to \re$ for which the Fermi Golden Rule
${\mathcal F}_{q,m,\lambda}$ is valid for each $m \in {\mathbb Z}$
and $q > m_-$, is dense in ${\mathcal E}_{\nu}$.
    \end{theorem}

{\bf 5.2.} For ${\mathcal  H}$ a subspace of $L^r(\re)$, let us
introduce the space
\begin{equation}\label{defhdagger}
{\mathcal  H}^\dagger:= \{ \omega \in {\mathcal S}(\re)  \; | \;
\forall \chi \in {\mathcal H}, \; \int_\re \omega(x)
\chi(x)dx=0\}.
\end{equation}
Clearly, if  $C^{\infty}_b(\re)\subset {\mathcal  H}$ then
${\mathcal H}^\dagger= \{0\}$. This property holds yet if $
H^\infty(S_\theta)\subset {\mathcal  H}$ where
$H^\infty(S_\theta)$ is the set of smooth bounded functions on
$\re$ admitting analytic extension on $S_\theta $. It is enough to
note that if $\omega(x_0)\neq 0$ then for $C$ sufficiently large,
the function $\chi(x_3):= e^{-C(x_3-x_0)^2}$ is in
$H^\infty(S_\theta)$ and satisfies $\int_\re \omega(x) \chi(x)dx
\neq 0$.

\begin{theorem} \label{t32a}
Assume that $v_0$ satisfies the hypotheses of Theorem \ref{t32}.
Let $p \geq 1$, $r\geq 1$,  $\delta \in \re$.\\
Let ${\mathcal  H}$ be a Banach space contained in  $L^r(\re)$
such that the injection ${\mathcal  H} \hookrightarrow L^r(\re)$
is
continuous, and  ${\mathcal  H}^\dagger= \{0\}$.\\
Then the set of real perturbations $V : \re_+ \times \re \to \re$,
for which the Fermi Golden Rule ${\mathcal F}_{q,m,\lambda}$ is
valid for each $m \in {\mathbb Z}$ and $q > m_-$, is dense in
$L^p(\re_+,  \langle \varrho \rangle^{\delta}\varrho d\varrho;
{\mathcal  H})$.
    \end{theorem}
    \begin{proof}
    As in the case of Theorems \ref{t32} -- \ref{t32aa},
    we will prove that the set of perturbations $V$ for which the
    integral \eqref{pp60}
 does not vanish for each $m \in {\mathbb Z}$ and $q > m_-$, is dense in
$L^p(\re_+,  \langle \varrho \rangle^{\delta}\varrho d\varrho;
{\mathcal  H})$.\\
Let
$${\mathcal M}_{q,m,\lambda}:=
\{ V \in L^p(\re_+, \langle \varrho \rangle^{\delta}\varrho
d\varrho; {\mathcal  H}); \; I_{q,m,\lambda}(V)\neq 0 \}.
$$
Since for $1/p' + 1/p =1$ and  $1/r' + 1/r =1$ we have
$$
\vert  I_{q,m,\lambda}(V) \vert \leq \|
\varphi_{q-1,m}\varphi_{q,m}\|_{ L^{p'}(\re_+,  \langle \varrho
\rangle^{-\delta}\varrho d\varrho; \re)} \; \| \omega
\|_{L^{r'}(\re)} \; \|  V \|_{ L^{p}(\re_+, \langle \varrho
\rangle^{\delta}\varrho d\varrho;L^{r}(\re)  )},
    $$
the continuity of the injection ${\mathcal  H} \hookrightarrow
L^r(\re)$ implies that ${\mathcal M}_{q,m,\lambda}$ is an open
subset of the Banach space $L^{p}(\re_+,  \langle \varrho
\rangle^{\delta} \varrho d\varrho; {\mathcal  H})$. Then according
to the Baire lemma, we have only to check that each ${\mathcal
M}_{q,m,\lambda}$ is dense in $L^{p}(\re_+, \langle \varrho
\rangle^{\delta} \varrho d\varrho; {\mathcal H})$. Let $V \in
L^{p}(\re_+,  \langle \varrho \rangle^{\delta} \varrho d\varrho;
{\mathcal  H})\setminus
{\mathcal M}_{q,m,\lambda}$.\\
Since $0 \neq \omega = \overline{\omega} \in {\mathcal S}(\re)$,
the assumptions on ${\mathcal H}$ imply the existence of a $\Phi
\in {\mathcal H}$ such that
$$
\int_{\re}  \omega(x_3) \Phi(x_3) dx_3 \neq 0.
$$
Moreover, $\varrho \mapsto  \varphi_{q-1,m}(\varrho)
\varphi_{q,m}(\varrho)\varrho$
 is a product of  polynomial and exponential functions. Then there exists
 $\varrho_0 \in \re_+$ such that $ \varphi_{q-1,m}(\varrho_0) \varphi_{q,m}(\varrho_0)\varrho_0\neq 0$,  and for $\chi_0$ supported near $\varrho_0$, we have
$$ \int_0^{\infty}  \varphi_{q-1,m}(\varrho) \varphi_{q,m}(\varrho)
 \chi_0(\varrho) \varrho d\varrho \neq 0.$$
Consequently, $\left\{V(\varrho,x_3) + \frac{1}{n} \chi_0(\varrho)
\Phi(x_3)\right\}_{n \in {\mathbb N}}$  is a sequence of functions
in ${\mathcal M}_{q,m,\lambda}$ tending to $V$ in  $L^{p}(\re_+,
\langle \varrho \rangle^{\delta} \varrho d\varrho;  {\mathcal
H})$.
\end{proof}

\section{Singularities of the spectral shift function}
\setcounter{equation}{0}

{\bf 6.1.} Suppose that $v_0$ satisfies \eqref{31}. Assume
moreover that the perturbation $V : \rt \to \re$ satisfies
    \eqref{33} with $\mpe > 2$ and $m_3 =m_0> 1$.
Then the multiplier by $V$ is a relatively trace-class
perturbation of $H$. Hence, the spectral shift function (SSF)
$\xi(\cdot; H + V, H)$ satisfying the Lifshits-Krein trace formula
$$
{\rm Tr} (f(H + V) - f(H)) = \int_{\re} f'(E) \xi(E; H + V, H) dE,
\quad f \in C_0^{\infty}(\re),
$$
and normalized by the condition $\xi(E; H + V, H)=0$ for $E < \inf
\sigma(H + V)$, is well-defined as an element of $L^1(\re; \langle
 E \rangle^{-2}dE)$ (see \cite{lif}, \cite{kr}).\\
If $E <  \inf \sigma(H)$, then the spectrum of $H + V$ below $E$
could be at most discrete, and for almost every $E < \inf
\sigma(H)$ we have
    $$
\xi(E; H + V, H) = - {\rm rank}\,{\mathbb P}_{(-\infty,E)}(H + V).
$$
 On the other hand,  for almost every
$E \in \sigma_{\rm ac}(H) = [0,\infty)$, the SSF $\xi(E; H + V,
H)$ is related to the scattering determinant ${\rm det}\;S(E; H +
V, H)$ for the pair $(H + V, H)$ by {\em the Birman-Krein} formula
$$
{\rm det}\;S(E; H + V, H) = e^{-2\pi i \xi(E; H + V, H)}
$$
(see \cite{bk}). \\
Set
 $$
 {\mathcal Z}: = \{E \in \re | E = 2bq + \mu, \; q \in
 {\mathbb Z}_+, \; \mu \in \sigma_{\rm disc}(\hpa) \; {\rm or} \; \mu = 0\}.
 $$
 Arguing as in the proof of \cite[Proposition 2.5]{BPR}, we can
 easily check the validity of the following
 \begin{pr} \label{p31}
 Let $v_0$ and $V$ satisfy \eqref{31}, and \eqref{33} with $\mpe > 2$ and $m_0 = m_3 > 1$.
 Then the SSF
 $\xi(\cdot;H + V, H)$ is bounded on every compact subset of
 $\re\setminus{\mathcal Z}$, and is continuous on $\re\setminus({\mathcal Z} \cup
 \sigma_p(H+V))$, where $\sigma_p(H+V)$ denotes the set of the eigenvalues of the operator $H+V$.
 \end{pr}
 In what follows we will assume in addition that
  \bel{32}
0 \leq V({\bf x}), \quad {\bf x} \in \rt,
    \ee
    and will consider the operators $H \pm V$ which are sign-definite
    perturbations of the operator $H$.
 The goal of this section is to investigate the asymptotic
 behaviour of the SSF $\xi(\cdot;H \pm V, H)$ near the energies
 which are eigenvalues of $H$ of infinite multiplicity. More precisely,
 if \eqref{ser7} holds true, and $\lambda \in \sigma_{\rm disc}(\hpa)$, we
 will study the asymptotics as $\eta \to 0$ of
 $\xi(2bq + \lambda + \eta;H \pm V, H)$, $q \in {\mathbb Z}_+$,
 being fixed. \\
 Let $T$ be a compact self-adjoint operator. For $s>0$ denote
 $$
 n_{\pm}(s;T) : = {\rm rank}\,{\mathbb P}_{(s,\infty)}(\pm T), \quad
 n_*(s;T) : = n_{+}(s;T) + n_{-}(s;T).
 $$
Put
$$
U(\xp) : = \int_{\re} V(\xp,x_3) \psi(x_3)^2 dx_3, \quad \xp \in
\rd,
$$
the eigenfunction $\psi$ being defined in \eqref{ser2}.

\begin{theorem} \label{t31}
Let $v_0$ and $V$ satisfy \eqref{31}, \eqref{33} with $\mpe > 2$
and $m_0 = m_3 > 1$, and \eqref{32}. Assume that \eqref{ser7}
holds true, and $\lambda \in \sigma_{\rm disc}(\hpa)$. Fix $q \in
{\mathbb Z}_+$. Then for each $\varepsilon \in (0,1)$ we have
    \bel{ssf31}
    n_+((1+\varepsilon) \eta; p_q U p_q) + O(1) \leq
    \pm \xi(2bq + \lambda \pm \eta; H \pm V, H) \leq
    n_+((1-\varepsilon) \eta; p_q U p_q) + O(1),
    \ee
    \bel{ssf32}
\xi(2bq + \lambda \mp \eta; H \pm V, H) = O(1),
    \ee
as $\eta \downarrow 0$.
\end{theorem}
Applying the well known results on the spectral asymptotics for
compact Berezin-Toeplitz operators $p_q U p_q$ (see \cite{r0},
\cite{rw}), we obtain the following
\begin{follow} \label{f41}
    Assume the hypotheses of Theorem \ref{t31}. \\
    (i)
    Suppose that $U \in C^1(\rd)$, and
$$
U(\xp) = u_0(\xp/|\xp|) |\xp|^{-\alpha}(1 + o(1)), \quad |\xp| \to
\infty,
$$
$$
|\nabla U(\xp)| \leq C_1 \langle \xp \rangle^{-\alpha-1}, \quad
\xp \in \rd,
$$
where $\alpha > 2$, and $u_0$ is a continuous function on
${\mathbb S}^1$ which is non-negative and does not vanish
identically. Then we have
$$
\xi(2bq + \lambda \pm \eta; H \pm V,H) = \pm \frac{b}{2\pi}
\left|\left\{\xp \in \rd | U(\xp)
> \eta\right\}\right|\;(1 + o(1)) =
$$
$$
\pm \eta^{-2/\alpha} \frac{b}{4\pi} \int_{{\mathbb S}^1}
u_0(s)^{2/\alpha}ds \, (1 + o(1)), \quad \eta \downarrow 0,
$$
where $|.|$ denotes the Lebesgue measure.\\
    (ii) Let $U \in L^{\infty}(\rd)$. Assume that
$$
\ln{U(\xp)} = -\mu |\xp|^{2\beta}(1 + o(1)), \quad |\xp| \to
\infty,
$$
for some $\beta \in (0,\infty)$,  $\mu \in (0,\infty)$.
  Then we have
$$
\xi(2bq + \lambda \pm \eta; H \pm V,H) = \pm \varphi_{\beta}(\eta)
\; (1 + o(1)), \quad \eta \downarrow 0, \quad \beta \in
(0,\infty),
$$
where
$$
\varphi_{\beta}(\eta) : = \left\{
\begin{array} {l}
\frac{b}{2\mu^{1/\beta}} |\ln{\eta}|^{1/\beta} \quad {\rm if}
\quad 0
< \beta < 1,\\
\frac{1}{\ln{(1+2\mu/b)}}|\ln{\eta}| \quad {\rm if} \quad
\beta = 1, \\
\frac{\beta}{\beta - 1}(\ln|\ln{\eta}|)^{-1}|\ln{\eta}| \quad {\rm
if} \quad 1 < \beta < \infty,
\end{array}
\right. \quad \eta \in (0,e^{-1}).
$$
     (iii) Let $U \in L^{\infty}(\rd)$. Assume that the
support of $U$ is compact, and that there exists a constant $C>0$
such that $U \geq C$ on an open non-empty subset of $\rd$. Then we
have
$$
\xi(2bq + \lambda \pm \eta; H \pm V, H) = \pm
(\ln|\ln{\eta}|)^{-1}|\ln{\eta}| (1 + o(1)), \quad \eta \downarrow
0.
$$
\end{follow}
{\em Remarks}: (i) The threshold behaviour of the SSF for various
magnetic quantum Hamiltonians has been studied in \cite{fr} (see
also \cite{r1}, \cite{r2}), and recently in \cite{brs}. The
singularities of the SSF described in Theorem \ref{t31} and
Corollary \ref{f41} are of somewhat different nature since $2bq +
\lambda$ is an infinite-multiplicity eigenvalue, and not a
threshold in the continuous spectrum of the unperturbed operator. \\
(ii) By the strict mathematical version of the Breit-Wigner
representation for the SSF (see \cite{pz1}, \cite{pz2}), the
resonances for various quantum Hamiltonians could be interpreted
as the poles of the SSF. In \cite{bbr} a Breit-Wigner
approximation of the SSF near the Landau level was obtained for
the 3D Schr\"odinger operator with constant magnetic field,
perturbed by a scalar potential satisfying \eqref{33} with $\mpe >
2$ and $m_3 > 1$. Moreover, it was shown in \cite{bbr} that
typically the resonances accumulate at the Landau levels. It is
conjectured that the singularities of the SSF $\xi(\cdot;H \pm
V,H)$ at the points $2bq + \lambda$, $q \in \Z_+$, are due to
accumulation of resonances to these points. One simple motivation
for this conjecture is the fact that if $V$ is axisymmetric, then
the eigenvalues of the operators $p_q U p_q$, $q \in \Z_+$,
appearing in \eqref{ssf31} are equal exactly to the quantities
$\langle V \Phi_{q,m}, \Phi_{q,m}\rangle_{L^2(\re_+ \times \re;
\varrho d\varrho dx_3)}$, $m \geq -q$, occurring in \eqref{vin4}
and \eqref{x10}. We leave for a future work the detailed analysis
of the relation between the singularities of the SSF at the points
$2bq + \lambda$ and the eventual accumulation of resonances at
these points. Hopefully, in this future work we will also extend
our results of Sections 3 -- 5 to the case of non-axisymmetric
perturbations $V$. \\
(iii) As mentioned above, if $\lambda \in \sigma_{\rm
disc}(\hpa)$, then $\lambda$ is an isolated eigenvalue of $H$ of
infinite multiplicity. Set
$$
\lambda_- : = \left\{
\begin{array} {l}
\sup \{\mu \in \sigma(H), \quad \mu < \lambda\} \quad {\rm if}
\quad \lambda > \inf \sigma(H), \\
- \infty \quad {\rm if} \quad \lambda = \inf \sigma(H),
\end{array}
\right.
$$
$$
\lambda_+ : = \inf \{\mu \in \sigma(H), \quad \mu > \lambda\}.
$$
By Pushnitski's representation of the SSF (see \cite{p1}), and the
Birman-Schwinger principle for discrete eigenvalues in gaps of the
essential spectrum, we have
$$
\xi(\lambda - \eta; H - V; H) = - n_+(1; V^{1/2}
(H-\lambda+\eta)^{-1} V^{1/2}) =
$$
$$
- {\rm rank}\,{\mathbb P}_{(\lambda_-, \lambda - \eta)}(H-V) + O
(1), \quad \eta \downarrow 0,
$$
$$
\xi(\lambda + \eta; H + V; H) = n_-(1; V^{1/2}
(H-\lambda-\eta)^{-1} V^{1/2}) =
$$
$$
{\rm rank}\,{\mathbb P}_{(\lambda + \eta, \lambda_+)}(H+V) + O(1),
\quad \eta \downarrow 0.
$$
Then Theorem \ref{t31} and Corollary \ref{f41} imply that the
perturbed operator $H - V$ (resp., $H+V$) has an infinite sequence
of discrete eigenvalues accumulating to $\lambda$ from
the left (resp., from the right). \\

 {\bf 6.2.} This subsection contains some
preliminary results
needed for the proof of Theorem \ref{t31}.  \\
In what follows we denote by $S_1$ the trace class, and by $S_2$
the Hilbert-Schmidt class of compact operators, and by
$\|\cdot\|_j$ the norm in $S_j$, $j=1,2$.\\
Suppose that $\eta \in \re$ satisfies
    \bel{ssf20}
    0 < |\eta| < \min{\left\{2b+\lambda, \frac{1}{2} {\rm dist}\,\left(\lambda,
    \sigma(\hpa)\setminus\{\lambda\}\right)\right\}}.
    \ee
    Note that inequalities \eqref{ssf20} combined with \eqref{ser7}, imply
    $$
    \lambda + \eta \in (-2b,0), \quad \lambda + \eta \not \in \sigma(\hpa), \quad {\rm dist}\,(\lambda + \eta,
    \sigma(\hpa)) = |\eta|.
    $$
Set $P_j = p_j \otimes I_{\parallel}$, $j \in {\mathbb Z}_+$. For
$z \in {\mathbb C}_+ : = \{\zeta \in {\mathbb C} | {\rm Im}\,\zeta
> 0\}$, $j \in {\mathbb Z}_+$, and $W : = V^{1/2}$, put
$$
 T_j(z) : = W P_j (H - z)^{-1} W.
$$
\begin{pr} \label{p32}
Assume  the hypotheses of Theorem \ref{t31}. Suppose that
\eqref{ssf20} holds true.
    Fix $q \in {\mathbb Z}_+$. Let $j \in {\mathbb Z}_+$, $j \leq
    q$. Then the operator-norm limit
    \bel{ssf50}
    T_j(2bq + \lambda + \eta) =
    \lim_{\delta \downarrow 0} T_j(2bq + \lambda + \eta + i\delta)
    \ee
exists in ${\mathcal L}(L^2(\rt))$. Moreover, if $j<q$,
we have $T_j(2bq + \lambda + \eta) \in S_1$, and
    \bel{ssf8}
    \|T_j(2bq + \lambda + \eta)\|_1 = O(1), \quad \eta \to 0.
    \ee
\end{pr}
\begin{proof}
We have
    \bel{ssf6}
    T_j(z) = M (t_{\perp, j} \otimes t_{\parallel}(z - 2bj)) M, \quad z
    \in {\mathbb C}_+,
    \ee
where
$$
M: = W(\xp, x_3) \langle \xp \rangle^{\mpe/2} \langle x_3
\rangle^{m_3/2},
$$
$$
t_{\perp, l} : = \langle \xp \rangle^{-\mpe/2} p_l \langle \xp
\rangle^{-\mpe/2}, \quad l \in {\mathbb Z}_+,
$$
$$
t_{\parallel}(\zeta) : = \langle x_3 \rangle^{-m_3/2} (\hpa  -
\zeta)^{-1} \langle x_3 \rangle^{-m_3/2}, \quad \zeta \in {\mathbb
C}_+.
$$
Since the operators $M$ and $t_{\perp}$ are bounded, in order to
prove that the limit \eqref{ssf50} exists in ${\mathcal
L}(L^2(\rt))$, it suffices to show that the operator-norm limit
    \bel{ssf1}
    \lim_{\delta \downarrow 0} t_{\parallel}(2b(q-j) + \lambda +
    \eta + i\delta)
    \ee
exists in ${\mathcal L}(L^2(\re))$. If $j<q$, the limit in
\eqref{ssf1} exists due to the existence of the limit in
\eqref{35}. If $j=q$, the limit in \eqref{ssf1} exists just
because $\lambda + \eta \not \in \sigma(\hpa)$.\\
Further, set
    $$
    t_{\parallel, 0}(\zeta) : = \langle x_3 \rangle^{-m_3/2} (\hopa
- \zeta)^{-1} \langle x_3 \rangle^{-m_3/2}, \quad \zeta \in
\overline{{\mathbb C}}_+\setminus\{0\}.
    $$
    For $E = 2b(q-j) + \lambda + \eta$, from the resolvent  equation we deduce
    \bel{ssf2}
t_{\parallel}(E) = t_{\parallel, 0}(E)(I_{\parallel} -\tilde{M}
t_{\parallel}(E))
    \ee
where $\tilde{M} : = v_0(x_3) \langle x_3 \rangle^{m_3}$ is a
bounded multiplier. By \cite[Section 4.1]{BPR}, the operator
$t_{\parallel, 0}(E)$ with $E \in \re\setminus\{0\}$ is
trace-class, and we have
    \bel{ssf3}
\|t_{\parallel, 0}(E)\|_1 \leq \frac{c}{\sqrt{|E|}} (1 +
E_+^{1/4})
    \ee
with $c$ independent of $E$. \\
Assume $j<q$. Then \eqref{ssf2}, \eqref{ssf3}, and \eqref{36}
imply $t_{\parallel}(2b(q-j) + \lambda + \eta) \in S_1$, and
    \bel{ssf4}
    \|t_{\parallel}(2b(q-j) + \lambda + \eta)\|_1 = O(1), \quad
    \eta \to 0.
    \ee
Finally, for any $l \in {\mathbb Z}_+$ we have $t_{\perp, l} \in
S_1$, and
    \bel{ssf7}
    \|t_{\perp, l}\|_1 = \frac{b}{2\pi} \int_{\re^2} \langle \xp
    \rangle^{-m_\perp} d \xp
    \ee
(see e.g. \cite[Subsection 4.1]{BPR}). Bearing in mind the
structure of the operator $T_j$ (see \eqref{ssf6}) and the
boundedness of the operator $M$, we find that $T_j(2b(q-j) +
\lambda + \eta) \in S_1$, and due to \eqref{ssf4} and
\eqref{ssf7}, estimate \eqref{ssf8} holds true.
\end{proof}
Now set $P_q^+ : = \sum_{j=q+1}^{\infty} p_j$, $q \in {\mathbb Z}_+$,
the convergence of the series being understood in the strong
sense. For $z \in {\mathbb C}_+$ set
$$
T_q^+(z) : = W (H-z)^{-1} P_q^+W.
$$
\begin{pr} \label{p33}
Assume that $v_0$, $V$, and $\lambda$, satisfy the hypotheses of
Theorem \ref{t31}, and $\eta \in \re$ satisfies \eqref{ssf20}. Fix
$q \in {\mathbb Z}_+$. Then the operator-norm limit
    \bel{ssf10}
T_q^+(2bq + \lambda + \eta) = \lim_{\delta \downarrow 0}T_q^+(2bq
+ \lambda + \eta + i\delta)
    \ee
    exists in ${\mathcal L}(L^2(\rt))$. Moreover,
    $T_q^+(2bq + \lambda + \eta) \in S_2$, and
    \bel{ssf9}
    \|T_q^+(2bq + \lambda + \eta)\|_2 = O(1), \quad \eta \to 0.
    \ee
\end{pr}
\begin{proof}
Due to \eqref{ser7}, the operator-valued function $$\C_+ \ni z
\mapsto (H-z)^{-1}P_q^+ \mapsto {\mathcal L}(L^2(\re^3))$$ admits
an analytic continuation in $\{\zeta \in {\mathbb C} \, | \, {\rm
Re}\,\zeta < 2bq\}$. Since $\lambda + \eta < 0$, and $W$ is
bounded, we immediately find that the limit in \eqref{ssf10}
exists.  Evidently, the operator-valued function $\C_+ \ni z
\mapsto (H_0-z)^{-1}P_q^+ \mapsto {\mathcal L}(L^2(\re^3))$ also
admits an analytic continuation in $\{\zeta \in {\mathbb C} \, |
\, {\rm Re}\,\zeta < 2bq\}$, and for $E = 2bq + \lambda + \eta$ we
have
    \bel{ssf13}
T_{q}^+(E) = W(H_0-E)^{-1}P_q^+ (W - v_0 (H-E)^{-1}P_q^+ W).
    \ee
Arguing as in the proof of \cite[Proposition 4.2]{fr}, we obtain
$$
W(H_0-2bq - \lambda - \eta)^{-1}P_q^+  \in S_2,
$$
and
    \bel{ssf11}
    \|W(H_0-2bq - \lambda - \eta)^{-1}P_q^+\|_2 = O(1), \quad \eta \to
    0.
    \ee
    Since $\lambda < 0$, we have
    \bel{ssf12}
\|W - v_0 (H-2bq - \lambda - \eta)^{-1}P_q^+ W\| = O(1), \quad
\eta \to
    0.
    \ee
    Putting together \eqref{ssf13} and
    \eqref{ssf11}-\eqref{ssf12}, we obtain \eqref{ssf9}.
\end{proof}

{\bf 6.3.} In this subsection we prove Theorem \ref{t31}.  \\
Suppose that $\eta \in \re$ satisfies \eqref{ssf20}. Fix $q \in
{\mathbb Z}_+$. Set
$$
T(2bq + \lambda + \eta) : = T_q^-(2bq + \lambda + \eta) + T_q(2bq
+ \lambda + \eta) + T_q^+(2bq + \lambda + \eta),
$$
where
$$
T_q^-(2bq + \lambda + \eta) = \sum_{j<q} T_j(2bq + \lambda +
\eta).
$$
Note that the operators $T_q(2bq + \lambda + \eta)$ and
$T_q^+(2bq + \lambda + \eta)$ are self-adjoint. \\
 By
Pushnitski's representation of the SSF for sign-definite
perturbations (see \cite{p1}), we have
    \bel{ssf27a}
    \xi(2bq + \lambda + \eta; H \pm V, H) =
    \pm \frac{1}{\pi} \int_{\re} n_{\mp}(1, {\rm Re}\,T(2bq + \lambda +
    \eta) + s\,{\rm Im}\,T(2bq + \lambda +
    \eta)) \frac{ds}{1+s^2}.
    \ee
By \eqref{ssf27a} and the well-known Weyl inequalities, for each
$\varepsilon \in (0,1)$ we have
    $$
n_{\mp}(1 + \varepsilon ; T_q(2bq + \lambda + \eta)) -
R_{\varepsilon}(\eta) \leq \pm \xi(2bq + \lambda + \eta; H \pm V,
H) \leq
$$
    \bel{ssf27}
    n_{\mp}(1 - \varepsilon ; T_q(2bq + \lambda + \eta)) +
R_{\varepsilon}(\eta)
 \ee
 where
 $$
R_{\varepsilon}(\eta) : =
$$
$$
n_*(\varepsilon/3; {\rm Re}\, T_q^-(2bq + \lambda + \eta)) +
n_*(\varepsilon/3; T_q^+(2bq + \lambda + \eta)) +
\frac{3}{\varepsilon} \|T_q^-(2bq + \lambda + \eta)\|_1 \leq
$$
    \bel{ssf28}
\frac{6}{\varepsilon} \sum_{j<q} \|T_j(2bq + \lambda + \eta)\|_1 +
\frac{9}{\varepsilon^2} \|T^+_q(2bq + \lambda + \eta)\|_2^2 = O(1),
\quad \eta \to 0,
    \ee
due to Propositions \ref{p31} -- \ref{p32}.\\
Next, set
$$
\tau_q = W (p_q \otimes \ppar) W, \quad \tilde{T}_q(\lambda +
\eta) : = W (p_q \otimes (\hpa  - \lambda - \eta)^{-1}
(I_{\parallel} - \ppar)) W,
$$
provided that $\eta \in \re$ satisfies \eqref{ssf20}. Evidently,
$$
T_q(2bq + \lambda + \eta) = -\eta^{-1} \tau_q +
\tilde{T}_q(\lambda + \eta).
$$
Applying again the Weyl inequalities, we get
$$
n_{\pm}(s(1+\varepsilon)|\eta|; ({\rm sign}\,\eta)\tau_q) -
n_*(s\varepsilon; \tilde{T}_q(\lambda + \eta)) \leq n_{\mp}(s;
T_q(2bq + \lambda + \eta)) \leq
$$
    \bel{ssf29}
n_{\pm}(s(1-\varepsilon)|\eta|; ({\rm sign}\,\eta) \tau_q) +
n_*(s\varepsilon; \tilde{T}_q(\lambda + \eta))
    \ee
for each $s>0$, $\varepsilon \in (0,1)$, and $\eta$ satisfying
\eqref{ssf20}. Note that since $p_q U p_q \geq 0$ we have
    \bel{ssf26}
n_{\pm}(s |\eta|; ({\rm sign}\,\eta)\,\tau_q) = \left\{
\begin{array} {l}
n_+(s |\eta|; p_q U p_q) \quad {\rm if} \quad \pm
\eta > 0, \\
0 \quad {\rm if} \quad \pm \eta < 0,
    \end{array}
    \right.
    \ee
for every $s>0$. Further,
    \bel{ssf25}
    \tilde{T}_q(\lambda + \eta) = M (t_{\perp, q} \otimes \tilde{t}_{\parallel}(\lambda + \eta))
    M
    \ee
where
$$
\tilde{t}_{\parallel}(\lambda + \eta) : = \langle x_3
\rangle^{-m_3/2} (\hpa  - \lambda - \eta)^{-1}
(I_{\parallel}-\ppar) \langle x_3 \rangle^{-m_3/2}.
$$
Obviously,
    \bel{ssf22}
 \|\tilde{t}_{\parallel}(\lambda + \eta)\| \leq
   \|(\hpa  - \lambda - \eta)^{-1} (I_{\parallel} - \ppar)\| = O(1), \quad \eta \to 0.
    \ee
On the other hand, similarly to \eqref{ssf2} we have
    \bel{ssf21}
\tilde{t}_{\parallel}(\lambda + \eta) = t_{\parallel, 0}(\lambda +
\eta)(I_{\parallel} -\tilde{M} \tilde{t}_{\parallel}(\lambda +
\eta))    - \langle x_3 \rangle^{-m_3/2} (\hopa  - \lambda -
\eta)^{-1} \ppar \langle x_3 \rangle^{-m_3/2}.
    \ee
Since $\ppar \langle x_3 \rangle^{-m_3/2}$ is a rank-one operator,
we have
$$
\|\langle x_3 \rangle^{-m_3/2} (\hopa  - \lambda - \eta)^{-1} \ppar
\langle x_3 \rangle^{-m_3/2}\|_1 \leq \|\langle x_3
\rangle^{-m_3/2} (\hopa  - \lambda - \eta)^{-1}\| \|\ppar \langle
x_3 \rangle^{-m_3/2}\|_2 \leq
$$
    \bel{ssf23}
    \int_{\re}\langle x \rangle^{-m_3} \psi(x)^2 dx \, |\lambda +
    \eta|^{-1} = O(1), \quad \eta \to 0.
    \ee
Putting together \eqref{ssf21}, \eqref{ssf3}, \eqref{ssf22}, and
\eqref{ssf23}, we get
    \bel{ssf24}
    \|\tilde{t}_{\parallel}(\lambda + \eta)\|_1 = O(1), \quad \eta \to
    0,
    \ee
which combined with \eqref{ssf25} and \eqref{ssf7} yields
    \bel{ssf30}
    n_{\pm}(s; \tilde{T}_q(\lambda + \eta)) \leq s^{-1} \|\tilde{T}_q(\lambda +
    \eta)\| = O(1), \quad \eta \to 0.
    \ee
Now \eqref{ssf31} -- \eqref{ssf32} follow from estimates
\eqref{ssf27} -- \eqref{ssf26}, and
\eqref{ssf30}.\\

{\bf Acknowledgements.} M. A. Astaburuaga and C. Fern\'andez were
partially supported by the Chilean {\em Laboratorio de An\'alisis
Estoc\'astico PBCT - ACT 13}.   Philippe Briet and Georgi Raikov
were partially supported by the CNRS-Conicyt Grant ``{\em
Resonances and embedded eigenvalues for quantum and classical
systems in exterior magnetic fields".} V. Bruneau was partially
supported by the French {\em ANR} Grant no. JC0546063. V. Bruneau
and G. Raikov were partially supported by the Chilean Science
Foundation {\em Fondecyt} under Grants 1050716 and 7060245.

\end{document}